\documentclass[11pt,reqno]{amsart}

\usepackage{xcolor}

\usepackage{amsmath, amsthm, amssymb}
\usepackage{amsfonts}
\usepackage[ansinew]{inputenc}
\usepackage[dvips]{epsfig}
\usepackage{graphicx}
\usepackage[english]{babel}
\usepackage{hyperref}

\usepackage{cite}
\usepackage{graphicx}
\usepackage{amscd}
\usepackage{bm}
\usepackage{enumerate}

\usepackage{verbatim}
\usepackage{hyperref}
\usepackage{amstext}
\usepackage{latexsym}
\theoremstyle{plain}
\newtheorem{theorem}{Theorem}[section]

\newtheorem{corollary}[theorem]{Corollary}
\newtheorem{proposition}[theorem]{Proposition}
\newtheorem{lemma}[theorem]{Lemma}
\newtheorem{remark}[theorem]{Remark}

\numberwithin{theorem}{section}
\numberwithin{equation}{section}

\newcommand{\cred}{\color{red}}

\newcommand{\average}{{\mathchoice {\kern1ex\vcenter{\hrule height.4pt
width 6pt depth0pt} \kern-9.7pt} {\kern1ex\vcenter{\hrule
height.4pt width 4.3pt depth0pt} \kern-7pt} {} {} }}

\def\R{\mathbb{R}}

\def\div{\text{div}}

\renewcommand{\a }{\alpha }

\renewcommand{\d}{\delta }

\newcommand{\D }{\Delta }

\newcommand{\e }{\varepsilon }

\newcommand{\G }{\Gamma}
\renewcommand{\l }{\lambda }

\newcommand{\n }{\nabla }
\newcommand{\vp }{\varphi }

\newcommand{\s }{\sigma }

\newcommand{\z }{\zeta}
\renewcommand{\th }{\theta }

\renewcommand{\O }{\Omega }

\newcommand{\ov}{\overline}

\newcommand{\be}{\begin{equation}}
\newcommand{\ee}{\end{equation}}

\newcommand{\de}{\partial}

\newcommand{\ti}{\widetilde}

\renewcommand{\k}{\kappa}

 \usepackage{mathrsfs}

\newcommand{\calH }{\mathcal{H}}
\newcommand{\calO }{\mathcal{O}}

\newcommand{\N}{\mathbb{N}}

\newcommand{\cA}{{\mathcal A}}

\newcommand{\cH}{{\mathcal H}}

\newcommand{\cK}{{\mathcal K}}

\newcommand{\cV}{{\mathcal V}}

\newcommand{\dist}{{\rm dist}}

\newcommand{\weak}{\rightharpoonup}

\newcommand{\eps}{\varepsilon}

\DeclareMathOperator{\id}{id}

\renewcommand{\epsilon}{\varepsilon}

\newcommand{\Ds}{ (-\D)^s}

\begin{document}
\title{A fractional Hadamard formula and applications}

\author[]
{Sidy Moctar Djitte${}^{1,2}$, Mouhamed Moustapha Fall${}^1$, Tobias Weth${}^2$}

%\address{Goethe-Universit\"{a}t Frankfurt, Institut f\"{u}r Mathematik.
%Robert-Mayer-Str. 10, D-60629 Frankfurt, Germany.}

\address{${}^1$African Institute for Mathematical Sciences in Senegal (AIMS Senegal), 
KM 2, Route de Joal, B.P. 14 18. Mbour, S\'en\'egal.}

\address{${}^2$Goethe-Universit\"{a}t Frankfurt, Institut f\"{u}r Mathematik.
Robert-Mayer-Str. 10, D-60629 Frankfurt, Germany.}

\email{djitte@math.uni-frankfurt.de, sidy.m.djitte@aims-senegal.org}
\email{weth@math.uni-frankfurt.de}
\email{mouhamed.m.fall@aims-senegal.org}

\maketitle

 \begin{abstract}
   \noindent 
We derive a shape derivative formula for the family of principal Dirichlet eigenvalues $\l_s(\Omega)$ of the fractional Laplacian $(-\Delta)^s$ associated with bounded open sets $\Omega \subset \R^N$ of class $C^{1,1}$. This extends, with a help of a new approach, a result in \cite{dal} which was restricted to the case $s=\frac{1}{2}$. As an application, we consider the maximization problem for $\lambda_s(\O)$ among annular-shaped domains of fixed volume of the type $B\setminus \ov B'$, where $B$ is a fixed ball and $B'$ is ball whose position is varied within $B$. We prove that $\lambda_s(B\setminus \ov B')$ is maximal when the two balls are concentric. Our approach also allows to derive similar results for the fractional torsional rigidity. More generally, we will characterize one-sided shape derivatives for best constants of a family of subcritical fractional Sobolev embeddings.
 \end{abstract}

\maketitle

\section{Introduction}\label{S:intro}

Let $s\in (0,1)$ and $\O \subset \R^N$ be a bounded open set. The present paper is devoted to the study of best constants $\lambda_{s,p}(\O)$ in the family of subcritical Sobolev inequalities
\begin{equation}
  \label{eq:sobolev-ineq-main}
\lambda_{s,p}(\O) \|u\|_{L^p(\Omega)}^2 \le   [u]_{s}^2 \qquad \text{for all $u \in \cH^s_0(\Omega)$,}
\end{equation}
where $p\in [1, \frac{2N}{N-2s})$ if $2s<N$ and $p\in [1, \infty)$ if $2s\geq N=1$. Here, the Sobolev space $\cH^s_0(\Omega)$ is given as completion of $C^\infty_c(\Omega)$ with respect to the norm $[\,\cdot\,]_{s}$ defined by
\be \label{eq:def-bNs}
[u]_{s}^2=  \frac{ c_{N,s}}{2}\int_{\R^{N}}\! \int_{\R^{N}}\!\frac{(u(x)-u(y))^2}{|x-y|^{N+2s}}dx dy \quad \qquad \text{with} \quad c_{N,s}= \pi^{-\frac{N}{2}}s 4^s\frac{\Gamma(\frac{N}{2}+s)}{\Gamma(1-s)}.
\ee
The normalization constant $c_{N,s}$ is chosen such that  $[u]_{s}^2=\int_{\R^N}|\xi|^{2s}|\hat{u}(\xi)|^2d\xi$ for $u \in \cH^s_0(\Omega)$, where $\hat{u}$ denotes the Fourier transform of $u$. The best (i.e., largest possible) constant in (\ref{eq:sobolev-ineq-main}) is given by 
\be \label{eq:def-lambda-sp}
 \lambda_{s,p}(\O):=\inf \left\{[u]_s^2\::\, u\in \cH^s_0(\Omega),\quad \|u\|_{L^p(\O)}=1  \right\}.
\ee
As a consequence of the subcriticality assumption on $p$ and the boundedness of $\Omega$, the space $\cH^s_0(\Omega)$ compactly embeds into $L^p(\O)$. Therefore a direct  minimization argument shows that $\lambda_{s,p}(\O) $ admits a nonnegative minimizer $u\in \cH^s_0(\Omega)$ with $\|u\|_{L^p(\Omega)}=1$. Moreover, every such minimizer solves, in the weak sense, the semilinear problem 
\begin{align}\label{eq:1.1}
\Ds u = \l_{s,p}(\O)u^{p-1}\quad\text{ in $\O$},\qquad \quad u = 0\quad\text{ in $\R^N\setminus\O$.}
\end{align} 
where $\Ds$ stands for the fractional Laplacian. It therefore follows from regularity theory and the strong maximum principle for $\Ds$ that $u$ is strictly positive in $\Omega$, see Lemma~\ref{reg-prop-minimizers} below. 
We recall that, for functions $\vp\in C^{1,1}_c(\R^N)$,  the fractional Laplacian is given by 
$$
\Ds\vp(x)=c_{N,s}\,PV \!\!\int_{\R^N} \frac{\vp(x)-\vp(y)}{|x-y|^{N+2s}}\, dy=\frac{c_{N,s}}{2}\int_{\R^N}\frac{2\vp (x)-\vp(x+y)-\vp(x-y)}{|y|^{N+2s}}\, dy.
$$
%and the normalization $c_{N,s}$ is such that the symbol of $\Ds$ is $|\xi|^{2s}$.\\
Of particular interest are the cases  $p=1$ and $p=2$ which correspond to the fractional torsion problem
\begin{align}\label{eq:1.1-ft}
\Ds u = \l_{s,1}(\O)\quad\text{ in $\O$},\qquad \quad u = 0\quad\text{ in $\R^N\setminus\O$,}
\end{align} 
and the eigenvalue problem
\begin{align}\label{eq:1.1-ev}
\Ds u = \l_{s,2}(\O)u\quad\text{ in $\O$},\qquad \quad u = 0\quad\text{ in $\R^N\setminus\O$,}
\end{align} 
associated with the first Dirichlet eigenvalue of the fractional Laplacian, respectively. In these cases, the minimization problem for $\l_{s,p}(\O)$ in \eqref{eq:def-lambda-sp} possesses a unique positive minimizer. Indeed, it is a well-known consequence of the fractional maximum principle that \eqref{eq:1.1-ft} admits a unique solution, and that \eqref{eq:1.1-ev} has a unique positive eigenfunction with $\|u\|_{L^2(\Omega)}=1$. Incidentally, the uniqueness of positive minimizers extends to the full range $1 \le p \le 2$, as we shall show in Lemma~\ref{uniqueness-extended} in
 the appendix of this paper.
%. . Moreover, the existence of two different positive normalized solutions of \eqref{eq:1.1-ev} would give rise to the existence of at least one normalized sign-changing eigenfunction $\tilde u$ of \eqref{eq:1.1-ev}, which then also is a solution of the minimization problem \eqref{eq:def-lambda-sp} for $p=2$. However, the fact that $\tilde u$ changes sign implies that $[|\tilde u|]_s < [\tilde u]_s$ for $s \in (0,1)$, which contradicts the minimization property of $\tilde u$.

Our first goal in this paper is to analyze the dependence of the best constants on the underlying domain $\Omega$. For this we shall derive a formula for a one-sided shape derivative of the map $\O\mapsto \l_{s,p}(\O)$. We assume from now on that $\Omega \subset \R^N$ is a bounded open set of class $C^{1,1}$, and we consider a family of deformations $\{\Phi_\eps\}_{\eps \in (-1,1)}$ with the following properties:
\be\label{eq:def-diffeom}
\begin{aligned}
&\text{$\Phi_\eps \in C^{1,1}(\R^N; \R^N)$ for $\eps  \in (-1,1)$, $\Phi_0= \id_{\R^N}$, and}\\
&\text{the map $(-1,1) \to C^{0,1}(\R^N,\R^N)$, $\eps \to \Phi_\eps$ is of class $C^2$.}  
 \end{aligned}
 \ee
{\cred  We note that \eqref{eq:def-diffeom} implies that $\Phi_\e:\R^N\to \R^N$ is a global diffeomorphism if $|\e|$ is small enough, see e.g. \cite[Chapter 4.1]{delfour-zolesio}. To clarify, we stress that 
 we only need the $C^2$-dependence of $\Phi_\eps$ on $\eps$ with respect to Lipschitz-norms, while $\Phi_\eps$ is assumed to be a $C^{1,1}$-function for $\eps \in (-1,1)$ to guarantee $C^{1,1}$-regularity of the perturbed domains $\Phi_\e(\O)$.}

From the variational characterization of  $ \l_{s,p}(\O)$ it is not difficult to see that the map $\e\mapsto \l_{s,p}(\Phi_\e(\O))$ is continuous. However, since   $\l_{s,p}(\O)$ may not have a unique positive minimizer, we cannot expect  this map to be differentiable.  
  We therefore rely on determining the right derivative of $\e\mapsto \l_{s,p}(\Phi_\e(\O))$ from which we derive differentiability whenever $ \l_{s,p}(\O)$  admits a unique positive minimizer, thereby extending the classical Hadamard shape derivative  formula for the first Dirichlet eigenvalue of the Laplacian $-\Delta$.

Throughout this paper, we consider a fixed function $\d \in C^{1,1}(\R^N)$ which coincides with the \textit{signed} distance function $ \dist(\cdot,\R^{N}\setminus \Omega)-\dist(\cdot,  \Omega)$ in a neighborhood of the boundary $\de\O$. We note here that, since we assume that $\Omega$ is of class $C^{1,1}$, the signed distance function is also of class $C^{1,1}$ in a neighborhood of $\de \O$ but not globally on $\R^N$. We also suppose that $\d$ is chosen with the property that $\d$ is positive in $\O$ and negative in $\R^N\setminus \overline\O$, as it is the case for the signed distance function.
  
   Our first main result is the following.
  \begin{theorem}\label{th:Shape-deriv}
Let $\l_{s,p}(\O)$ be given by  \eqref{eq:def-lambda-sp} and consider a family of deformations $\Phi_\eps$ satisfying \eqref{eq:def-diffeom}. Then the map $\e\mapsto \th(\e):=\l_{s,p}(\Phi_\e(\O))$    is right differentiable  at $\e=0$.  Moreover,
\be\label{eq:right-der-main-th}
\de_+\th(0)=\min \left\{ \G(1+s)^2 \int_{\de\O}(u/\d^s)^2 X \cdot \nu \,dx\::\, u\in \calH    \right\},
\ee
where $\nu$ denotes the interior unit normal on $\partial \Omega$, $\calH$ is the set of positive minimizers for  $  \l_{s,p}(\O)$    and 
$X := \partial_\e \big|_{\e=0} \Phi_\eps$.
\end{theorem}

Here the function $u/\d^s$ is defined on $\partial \Omega$ as a limit. Namely, for $x_0\in \partial \O$, the limit  
\begin{equation}
  \label{eq:limit-def-fractional-normal}
\frac{u}{\d^s}(x_0)=\lim_{\stackrel{x\to x_0 }{x\in \O}}\frac{u}{\d^s}(x)
\end{equation}
exists, as the function $u/\d^s$ extends to a function in $C^\a(\ov\O)$ for some $\a>0$, see \cite{RX}. In addition, the function $\d^{1-s}\n u$ also admits a Hölder continuous extension on $\overline \Omega$ satisfying $\d^{1-s}\n u\cdot \nu=s u/\d^s$ on $\de\O$, see \cite{FS-2019}. As a consequence, the expression $u/\d^s$,  restricted on $\de\O$,  plays the role of an inner fractional normal derivative. Note that, for $s = 1$, the limit on the RHS of (\ref{eq:limit-def-fractional-normal}) coincides with the classical inner normal derivative of $u$ at $x_0$. 

We observe that the constant $\G(1+s)^2$ appears also in the  fractional Pohozaev identity, see e.g. \cite{RX-Poh}. This is, to some extend, not surprising at least in the classical case since   Pohozaev's identity  can be obtained using techniques of domain variation, see e.g. \cite{Wagner}. 

We also remark that one-sided derivatives naturally arise in the analysis of parameter-dependent minimization problems, see e.g. \cite[Section 10.2.3]{delfour-zolesio} for an abstract result in this direction. Related to this, they also appear in the analysis of the domain dependence of eigenvalue problems with possible degeneracy, see e.g. \cite{FW18} and the references therein.

A natural  consequence of Theorem \ref{th:Shape-deriv}  is that the map $\e\mapsto \th(\e)=\l_{s,p}(\Phi_\e(\O))$ is differentiable at $\eps = 0$ whenever $\l_{s,p}(\O)$ admits a unique positive minimizer. Indeed, applying Theorem \ref{th:Shape-deriv} to the map  $\e\mapsto \ti \th(\e):=\l_{s,p}(\Phi_{-\e}(\O))$ yields
$$
\de_-\th(0)=-\de_+\ti \th(0)=\max \left\{\G(1+s)^2 \int_{\de\O}(u/\d^s)^2 X\cdot \nu\,dx\::\, u\in \calH    \right\},
$$
where $\cH$ is given as in Theorem~\ref{th:Shape-deriv}. As a consequence, we obtain the following result.
\begin{corollary}\label{cor:1.2}
Let $\l_{s,p}(\O)$ be given by  \eqref{eq:def-lambda-sp} and consider a family of deformations $\Phi_\eps$ satisfying \eqref{eq:def-diffeom}. Suppose that $ \l_{s,p}(\O)$ admits a unique positive minimizer $u\in \cH^s_0(\Omega)$. 
 Then the map $\e\mapsto \th(\e)=\l_{s,p}(\Phi_\e(\O))$    is differentiable  at $\e=0$.  Moreover
% %\lim_{\e\searrow 0}\frac{\l_{s,p}(\O_\e)-\l_{s,p}(\O)}{\e}:=
 \begin{equation}
   \label{eq:cor:1.2-formula}
 \th'(0)=\G(1+s)^2 \int_{\de\O}(u/\d^s)^2 X\cdot \nu\,dx,
 \end{equation}
where $X:=\de_\e\big|_{\e=0}\Phi_\eps$.
\end{corollary}
As mentioned earlier, $ \l_{s,p}(\O)$ admits a unique positive minimizer $u\in \cH^s_0(\Omega)$ for $1 \le p \le 2$, see Lemma~\ref{uniqueness-extended} in the appendix.  Therefore  Corollary \ref{cor:1.2}  extends, in particular, the classical Hadamard formula, for the first Dirichlet eigenvalue $\l_{1,2}(\O)$ of $-\Delta$, to the fractional setting. We recall, see e.g. \cite{Eh},  that the classical Hadamard formula is given by 
  \begin{align}\label{eq:1.2}
 \frac{d}{d\e}\Big|_{\e = 0}\l_{1,2}(\Phi_\e(\O)) = \int_{\de\O} |\n u|^2 X\cdot \nu \,dx.
\end{align} 
{\cred An analogue of Corollary \ref{cor:1.2} for the case of the local $r$-Laplace operator was obtained in  \cite{ ML,CFR}. 
We also point out that, prior to this paper, a Hadamard formula in the fractional setting of the type (\ref{eq:cor:1.2-formula}) was obtained in  \cite{dal} for the special case  $p=1$, $s= \frac{1}{2}$, $N = 2$ and $\O$ of class $C^\infty$. We are not aware of any other previous work related to Theorem~\ref{th:Shape-deriv} or \ref{cor:1.2} in the fractional setting.} 

 Our next result provides a characterization of constrained local minima of $\l_{s,p}$. Here and in the following, we call a bounded open subset $\O$ of class $C^{1,1}$ a constrained local minimum for $\l_{s,p}$ if for all families of deformations $\Phi_\eps$ satisfying \eqref{eq:def-diffeom} and the volume invariance condition $|\Phi_\eps(\Omega)|= |\Omega|$ for $\eps \in (-1,1)$, there exists $\e_0 \in (0,1)$ with $\l_{s,p}(\Phi_\e(\O)) \geq \l_{s,p}(\O)$ for $\e\in (-\e_0,\e_0)$. Our classification result reads as follows.

 \begin{corollary}\label{cor:1.3} 
Let $p\in \{1\}\cup [2,\infty)$.
If an open subset $\O$ of $\R^N$ of class $C^{3}$ is a volume constrained local minimum for $\O\mapsto \l_{s,p}(\O)$, then $\O$ is a ball.
\end{corollary}
 Corollary \ref{cor:1.3}  is a consequence of Theorem \ref{th:Shape-deriv}, from which   we derive that if $\O$ is a constrained local minimum then any element   $u\in \calH$ satisfies the overdetermined condition $u/\d^s\equiv constant$ on $\de\O$. Therefore by the rigidity result in \cite{JW17} we find that $\O$ must be a ball. {\cred We point out that we are not able to include the case $p \in (1,2)$ in  Corollary \ref{cor:1.3}, since the rigidity result in \cite{JW17} is based on the moving plane method and therefore requires the nonlinearity in \eqref{eq:1.1} to be Lipschitz. The case $p \in (1,2)$ therefore remains an open problem in Corollary \ref{cor:1.3}.} \\
We note that the authors in \cite{dal} considered the shape minimization problem  for  $\l_{s,p}(\Omega)$ in the case $p=1$, $s= \frac{1}{2}$, $N = 2$ among domains $\O$ of class $C^\infty$ of fixed volume. They showed in \cite{dal} that such minimizers are discs.
 
Next we consider the optimization problem of $\O\mapsto \lambda_{s,p}(\O)$ for $p\in \{1,2\}$ and $\O$ a punctured ball, with the hole having the shape of ball.   We show that, as the hole moves in $\O$ then  $\lambda_{s,p}(\O)$ is maximal  when the two balls are concentric. In the local case $s=1$ and $N=2$, this is a classical result by Hersch \cite{Hersch}. For subsequent generalizations in the case of the local problem, see \cite {LL14,A-C,Kesavan}. 
\begin{theorem}\label{th:1.2}
Let $p\in \{1,2\}$, $B_1(0)$ be the  unit centered ball and $\tau \in (0,1)$. Define 
$$
\cA:=\{a\in B_{1}(0)\,:\, B_\tau(a) \subset B_1(0)\}.
$$
Then the  map $\cA\to \R$, $\,a\mapsto\lambda_{s,p}(B_1(0)\setminus \ov{ B_\tau(a)})$ takes its maximum at $a=0$.
\end{theorem}  
  The proof of Theorem  \ref{th:1.2} is inspired by the argument given in \cite {LL14,Kesavan} for the local case $s=1$. It uses the fractional Hadamard formula in Corollary \ref{cor:1.2} and maximum principles for anti-symmetric functions. Our proof also shows that the map $ a\mapsto\lambda_{s,p}(B_1(0)\setminus \ov{ B_\tau(a)})$ takes its minimum when the boundary of the  ball $B_{\tau}(a)$ touches the one of $B_{1}(0)$, see Section~\ref{S:proofs} below.
  
The proof of Theorem \ref{th:Shape-deriv} is based on the use of test functions in the variational characterization of $\l_{s,p}(\O)$ and $\l_{s,p}(\Phi_\e(\O))$. The general strategy is inspired by the direct approach in \cite{FW18}, which is related to a Neumann eigenvalue problem on manifolds. In the case of $\l_{s,p}(\Phi_\e(\O))$, it is important to make a change of variables so that  $\l_{s,p}(\Phi_\e(\O))$ is determined by minimizing an $\e$-dependent family of seminorms among    functions $u\in \cH^s_0(\Omega)$, see Section \ref{S:Not-per} below. An obvious choice of  test functions are minimizers $u$ and $v_\e$ for $\l_{s,p}( \O)$ and  $\l_{s,p}(\Phi_\e(\O))$, respectively. However, due to the fact that 
$u$ is only of class $C^s$ up to the boundary, we cannot obtain a boundary integral term directly from the divergence theorem. 
In particular, the integration by parts formula given in \cite[Theorem 1.9]{RX-Poh} does not apply to general vector fields $X$ which appear in \eqref{eq:right-der-main-th}. Hence, we need to replace $u$ with $\z_k u$, where $\z_k$ is a cut-off function vanishing in a $\frac{1}{k}$-neighborhood of $\de\O$. This leads to upper and lower estimates of $\l_{s,p}(\Phi_\e(\O))$ up to order $o(\e)$, where the first order term is given by an integral involving $\Ds(\z_k u)$ and $\n (\z_k u)$.  We refer the reader to Section~\ref{S:Onse-side} below for more precise information. 
 A highly nontrivial task is now to pass to the limit as $k\to \infty$ in order to get boundary integrals involving $\psi:= u/\d^s$. This is the most difficult part of the paper. We refer to Proposition~\ref{lem:lim-ok} and Section~\ref{s:proof-Lemma-conv} below for more details.    
 
The paper is organized as follows.  In Section \ref{S:Not-per}, we provide preliminary results on convergence properties of integral functional, inner approximations of functions in $\cH^s_0(\Omega)$ and on properties of minimizers of \eqref{eq:def-lambda-sp}.
In Section~\ref{sec:doma-pert-assoc}, we introduce notation related to domain deformations and related quantities.  In Section \ref{S:Onse-side} we establish a preliminary variant of Theorem~\ref{th:Shape-deriv}, which is given in Proposition~\ref{prop:shape-deriv}. In this variant, the constant $\G(1+s)^2$ in \eqref{eq:right-der-main-th}
is replaced by an implicitly given value which still depends on cut-off data. The proofs of the main results, as stated in this introduction, are then completed in Section \ref{S:proofs}. Finally, Section \ref{s:proof-Lemma-conv} is devoted to the proof of the main technical ingredient of the paper, which is given by Proposition \ref{lem:lim-ok}.\\ 

 \textbf{Acknowledgements:} This work is supported by DAAD and BMBF (Germany) within the project 57385104. The authors would like to thank Sven Jarohs for helpful discussions. M.M Fall's work is supported by the Alexander von Humboldt foundation.

\section{Notations and preliminary results}\label{S:Not-per}

Throughout this section, we fix a bounded open set $\Omega \subset \R^N$. As noted in the introduction, we define the space $\cH^s_0(\Omega)$ as completion of $C^\infty_c(\Omega)$ with respect to the norm $[\,\cdot\,]_{s}$ given in \eqref{eq:def-bNs}. Then $\cH^s_0(\Omega)$ is a Hilbert space with scalar product
$$
(u,v) \mapsto [u,v]_{s}=  \frac{ c_{N,s}}{2}\int_{\R^{N}}\! \int_{\R^{N}}\!\frac{(u(x)-u(y))(v(x)-v(y))}{|x-y|^{N+2s}}dx dy,
$$
where $c_{N,s}$ is given in \eqref{eq:def-bNs}. It is well known and easy to see that $\cH^s_0(\Omega)$ coincides with the closure of $C^\infty_c(\Omega)$ in the standard fractional Sobolev space $H^s(\R^N)$. Moreover, if $\Omega$ has a continuous boundary,
then $\cH^s_0(\Omega)$ admits the highly useful characterization
\begin{equation}
  \label{useful-characterization}
\cH^s_0(\Omega)= \bigl \{ w \in L^1_{loc}(\R^N)\::\: [w]_s^2 < \infty,\quad w \equiv 0\; \text{on $\R^N \setminus \Omega$} \bigr\},  
\end{equation}
see e.g. \cite[Theorem 1.4.2.2]{G11}. We start with an elementary but useful observation.
\begin{lemma}\label{lem:v_ez_k-ok-prelim}
Let $\mu\in L^\infty(\R^N\times \R^N)$, and let $(v_k)_k$ be a sequence in $\cH^s_0(\Omega)$ with $v_k \to v$ in $\cH^s_0(\Omega)$ as $k \to \infty$. Then we have
$$
\lim_{k\to\infty}\int_{\R^{2N}} \frac{(v_k(x)-v_k(y))^2\mu(x,y)}{|x-y|^{N+2s}} dxdy = \int_{\R^{2N}} \frac{(v(x)-v(y))^2\mu(x,y)}{|x-y|^{N+2s}} dxdy.  \\
$$
\end{lemma}

\begin{proof}
We have   
\begin{align*}
  &\Bigl|\int_{\R^{2N}} \frac{(v_k(x)-v_k(y))^2-(v(x)-v(y))^2\mu(x,y)}{|x-y|^{N+2s}} dxdy\Bigr|\\
  &\le 
\|\mu\|_{L^\infty} \int_{\R^{2N}} \frac{|(v_k(x)-v_k(y))^2-(v(x)-v(y))^2|}{|x-y|^{N+2s}} dxdy, 
\end{align*}
where 
\begin{align*}
&\int_{\R^{2N}} \frac{|(v_k(x)-v_k(y))^2-(v(x)-v(y))^2|}{|x-y|^{N+2s}} dxdy\\ 
&=\int_{\R^{2N}} \frac{|[(v_k(x)-v(x))-(v_k(y)-v(y))][(v_k(x)+v(x))-(v_k(y)+v(y))]|}{|x-y|^{N+2s}} dxdy\\
&\le \frac{2}{c_{N,s}}[v_k-v]_s [v_k+v]_s\: \to\: 0 \qquad \text{as $k \to \infty$.}
\end{align*}
\end{proof}

Throughout the remainder of this paper, we fix $\rho\in C^\infty_c(-2,2)$ with $0 \le \rho \le 1$, $\rho\equiv 1$ on $(-1,1)$, and we define
\begin{equation}
  \label{eq:def-z}
  \z \in C^\infty(\R),\qquad \z(t)=1-\rho(t).
\end{equation}
Moreover, for $k \in \N$, we define the functions
\begin{equation}
  \label{eq:def-z-k-etc}
\rho_k,\: \z_k\: \in\: C^{1,1}(\R^N), \quad \qquad \rho_k(x)= \rho(k \delta(x)), \;\quad \z_k(x)= \z(k \delta(x)).
\end{equation}
We note that the function $\rho_k$ is supported in the $\frac{2}{k}$-neighborhood of the boundary, while the function $\z_k$ vanishes in the $\frac{1}{k}$-neighborhood of the boundary.

\begin{lemma}
\label{inner-approx-property}
Let $\Omega \subset \R^N$ be a bounded Lipschitz domain and let $u \in \cH^s_0(\Omega)$. Moreover, for $k \in \N$, let $u_k := u \zeta_k \in \cH^s_0(\Omega)$ denote inner approximations of $u$. Then we have 
$$
u_k \to u \qquad \text{in $\cH^s_0(\Omega)$.}
$$ 
\end{lemma}

\begin{proof}
In the following, the letter $C>0$ stands for various constants independent of $k$. Since $\rho_k= 1 -\z_k$, it suffices to show that 
\begin{equation}
  \label{eq:u-phi-r-claim}
u \rho_k  \in \cH^s_0(\Omega)\; \text{for $k$ sufficiently large}\qquad \text{and}\qquad [u \rho_k]_s \to 0 \quad \text{as $k \to \infty$.}
\end{equation}
For $\eps>0$, we put $A_\eps= \{x \in \Omega\::\: \delta(x)< \eps\}$.
Since $u \rho_k$ vanishes in $\R^N \setminus A_{\frac{2}{k}}$, $0 \le \rho_k \le 1$ on $\R^N$ and $|{\rho_k}(x)-{\rho_k}(y)| \le C \min \{k |x-y|,1\}$ for $x,y \in \R^N$, we observe that
	\begin{align}
	&\frac{1}{c_{N,s}} [{\rho_k} u]^2_s =\frac{1}{2}\int_{\R^N}\int_{\R^N} \frac{[u(x){\rho_k}(x)-u(y){\rho_k}(y)]^2}{|x-y|^{N+2s}} dydx \nonumber\\
        &= \frac{1}{2}\int_{A_{\frac{4}{k}}}\int_{A_{\frac{4}{k}}} \frac{[u(x){\rho_k}(x)-u(y){\rho_k}(y)]^2}{|x-y|^{N+2s}}\ dydx +\int_{A_{{\frac{2}{k}}}} u(x)^2 {\rho_k}(x)^2 \int_{\R^N \setminus A_{\frac{4}{k}}} |x-y|^{-N-2s}\ dydx\nonumber\\
&\le \frac{1}{2}\int_{A_{\frac{4}{k}}}\int_{A_{\frac{4}{k}}} \frac{\bigl[u(x)\bigl({\rho_k}(x)-{\rho_k}(y)\bigr)+ {\rho_k}(y)\bigl(u(x)-u(y)\bigr)\bigr]^2}{|x-y|^{N+2s}}\ dydx\nonumber\\
          &\quad+ C  \int_{A_{{\frac{2}{k}}}} u(x)^2  \dist(x,\R^N \setminus A_{\frac{4}{k}})^{-2s} dx            \nonumber\\
&\le \int_{A_{\frac{4}{k}}} u^2(x) \int_{A_{\frac{4}{k}}}\frac{({\rho_k}(x)-{\rho_k}(y))^2}{|x-y|^{N+2s}}\ dydx
                                                                                                                                + \int_{A_{\frac{4}{k}}}\int_{A_{\frac{4}{k}}}\frac{(u(x)-u(y))^2}{|x-y|^{N+2s}}dydx \nonumber
          \end{align}
\begin{align}
          &\quad+ C\int_{A_{{\frac{2}{k}}}} u(x)^2 \delta^{-2s}(x)dx\nonumber\\
          &\le C k^2 \int_{A_{\frac{4}{k}}} u^2(x)\int_{B_{\frac{1}{k}}(x)} |x-y|^{2-2s-N} dydx + C \int_{A_{\frac{4}{k}}} u^2(x)\int_{\R^N \setminus B_{\frac{1}{k}}(x)} |x-y|^{-N-2s} dydx\nonumber\\
&\quad + \int_{A_{\frac{4}{k}}}\int_{A_{\frac{4}{k}}}\frac{(u(x)-u(y))^2}{|x-y|^{N+2s}}dydx + C  \int_{A_{{\frac{2}{k}}}} u(x)^2  \delta^{-2s}(x) dx \nonumber\\
&\le C k^{2s} \int_{A_{\frac{4}{k}}} u^2(x)dx + \int_{A_{\frac{4}{k}}}\int_{A_{\frac{4}{k}}}\frac{(u(x)-u(y))^2}{|x-y|^{N+2s}}dydx + C  \int_{A_{{\frac{2}{k}}}} u(x)^2 \delta^{-2s}(x) dx\nonumber\\
&\le C \int_{A_{\frac{4}{k}}} u^2(x)\delta^{-2s}(x)dx + \int_{A_{\frac{4}{k}}}\int_{A_{\frac{4}{k}}}\frac{(u(x)-u(y))^2}{|x-y|^{N+2s}}dydx. \label{proof-inner-approximation-1}
	\end{align}
Now, since $\Omega$ has a Lipschitz boundary, using 
$
\int_{\R^N \setminus \Omega}|x-y|^{-N-2s}\,dy\sim \d^{-2s}(x)
$ see e.g \cite{HC}, we get 
$$
\int_{\Omega} u^2(x)\delta^{-2s}(x)dx \le C \int_{\Omega}u^2(x) \int_{\R^N \setminus \Omega}|x-y|^{-N-2s}\,dy dx \le C[u]_s^2, 
$$
and therefore
\begin{equation}
\label{proof-inner-approximation-2}  
\int_{A_{\frac{4}{k}}} u^2(x)\delta^{-2s}(x)dx \to 0 \qquad \text{as $k \to \infty$.}
\end{equation}
Moreover, since also  
$$
\int_{\Omega} \int_{\Omega}\frac{(u(x)-u(y))^2}{|x-y|^{N+2s}}dydx \le \frac{2}{c_{N,s}}[u]_s^2, 
$$
we have 
\begin{equation}
\label{proof-inner-approximation-3}  
\int_{A_{\frac{4}{k}}}\int_{A_{\frac{4}{k}}}\frac{(u(x)-u(y))^2}{|x-y|^{N+2s}}dydx \to 0 \qquad \text{as $k \to \infty$.}
\end{equation}
Combining (\ref{proof-inner-approximation-1}), (\ref{proof-inner-approximation-2}) and (\ref{proof-inner-approximation-3}), we obtain (\ref{eq:u-phi-r-claim}), as required.
\end{proof}

From now on, we fix a bounded $C^{1,1}$-domain $\Omega \subset \R^N$. We also let $$
C^s_0(\ov \O) = \left\lbrace w\in C^s(\ov\O): w = 0\quad\text{in}\quad \R^N\setminus\O \right\rbrace,
$$
and we recall the following regularity and positivity properties of nonnegative minimizers for $\lambda_{s,p}(\Omega)$ as defined in \eqref{eq:def-lambda-sp}. 

\begin{lemma}
\label{reg-prop-minimizers}
Let $u \in \cH_0^s(\Omega)$ be a nonnegative minimizer for $\lambda_{s,p}(\Omega)$. Then $u\in C^s_0(\ov \O) \cap C^\infty_{loc}(\Omega)$. Moreover, $\psi:=\frac{u}{\d^s} \in C^{\a}(\ov\O)$ for some $\alpha \in (0,1)$, and there exists a constant $c=c(N,s,\O,\a,p)>0$ with the property that  
\be \label{eq:Bnd-reg-v-eps}
\|\psi\|_{C^{\a}(\ov\O)}\leq c  
\ee 
and 
\be\label{eq:Grad-bnd-psi-eps}
| \n \psi (x)  | \leq c \d^{\a-1}(x)   \qquad\textrm{ for all $x\in \O$.}
\ee
Moreover, $\psi >0$ on $\overline \Omega$, so in particular $u>0$ in $\Omega$.
\end{lemma}

\begin{proof}
  By standard arguments in the calculus of variations, $u$ is a weak solution of (\ref{eq:1.1}). By \cite[Proposition 1.3]{XJV} we have that $u\in L^\infty(\O)$, and therefore the RHS of (\ref{eq:1.1}) is a function in $L^\infty(\O)$. Thus the regularity up to the boundary $u \in C^s_0(\ov \O)$
is proved in \cite{RX}, where also the $C^\alpha$-bound \eqref{eq:Bnd-reg-v-eps} for the function $\psi= \frac{u}{\d^s}$ is established for some $\alpha>0$. Moreover, \eqref{eq:Grad-bnd-psi-eps} is proved in \cite{FS-2019}. It also follows from (\ref{eq:1.1}), the strong maximum principle and the Hopf lemma for the fractional Laplacian that $\psi$ is a strictly positive function on $\overline \Omega$. In particular, $u>0$ in $\Omega$, Therefore $u \in C^\infty_{loc}(\O)$ follows by interior regularity theory (see e.g. \cite{RS16a}) and the fact that the function $t \mapsto t^{p-1}$ is of class $C^\infty$ on $(0,\infty)$.
\end{proof}

The computation of one-sided shape derivatives as given in Theorem~\ref{th:Shape-deriv} will be carried out in Section~\ref{S:Onse-side}, and it requires the following key technical proposition. Since its proof is long and quite involved,  we postpone the proof to Section \ref{s:proof-Lemma-conv} below.

\begin{proposition}\label{lem:lim-ok}
Let $X \in C^{0}(\overline \Omega,\R^N)$, let $u \in C^s_0(\ov \O)\cap C^1(\O)$, and assume that 
$\psi:=\frac{u}{\d^s}$ extends to a function on $\overline \Omega$ satisfying \eqref{eq:Bnd-reg-v-eps} and \eqref{eq:Grad-bnd-psi-eps}. Moreover, put $U_k:=u \z_k  \in C^{1,1}_c(\Omega)$, where $\z_k$ is defined in (\ref{eq:def-z-k-etc}). Then 
$$
\lim_{k\to \infty}\int_{\Omega}\n U_k \cdot X \Bigl( u   \Ds \z_k  -I(u,\z_k)\Bigr)\,dx=  -\k_s  \int_{\de\O} \psi^2 X\cdot \nu \, dx,
$$
where \be \label{eq:-kappa-s-implicit}
 \k_s:= - \int_{\R}h'(r) \Ds h(r)\, dr \qquad \text{with}\quad h(r):=r_+^s \z(r)=\max(r,0)^s\z(r)
\ee
and $\z$ given in (\ref{eq:def-z}), and where we use the notation
\be\label{eq:def-I-u-v}
  I(u,v)(x):= c_{N,s}\int_{\R^N}\frac{(u(x)-u(y))(v(x)-v(y))}{|x-y|^{N+2s}}\, dy
\ee
for $u\in C^s_c(\R^N)$, $v\in C^{0,1}(\R^N)$ and $x \in \R^N$.
\end{proposition}

\begin{remark}
\label{sec:notat-prel-results-remark}
The minus sign in the definition of the constant $\k_s$ in \eqref{eq:-kappa-s-implicit} might appear a bit strange at first glance. We shall see later that, defined in this way, $\kappa_s$ has a positive value. A priori it is not clear that the value of $\k_s$ does not depend on the particular choice of the function $\z$. This follows a posteriori once we have established in Proposition~\ref{prop:shape-deriv} below that this constant appears in Theorem~\ref{th:Shape-deriv}. This will then allow us to show that $\k_s = \frac{\Gamma(1+s)^2}{2}$ by applying the resulting shape derivative formula to a one-parameter family of concentric balls, see Section~\ref{S:proofs} below. A more direct, but somewhat lengthy computation of $\k_s$ is possible via the logarithmic Laplacian, which has been introduced in \cite{WC}.     
\end{remark}

\section{Domain perturbation and the associated variational problem}
\label{sec:doma-pert-assoc} 

%We fix a map $\Phi: (-1,1)\times \R^N\to \R^N$ satisfying \eqref{eq:def-diffeom}. Moreover, for $\e \in (-1,1)$, we write $\Phi_\eps(x)=\Phi(\e,x)$. 
Here and in the following,   we  define $\Omega_\eps: = \Phi_\eps(\Omega)$. In order to study the dependence of $\lambda_{s,p}(\Omega_\eps)$ on $\eps$, it is convenient to pull back the problem on the fixed domain $\Omega$ via a change of variables. For this we let $\textrm{Jac}_{\Phi_\e}$ denote the Jacobian determinant of the map $\Phi_\eps \in C^{1,1}(\R^N)$, and we define the kernels
\be \label{eq:def-K-eps}
K_\e(x,y):=  { c_{N,s}}\frac{\textrm{Jac}_{\Phi_\e}(x)\textrm{Jac}_{\Phi_\e}(y)}{|\Phi_\e(x)-\Phi_\e(y)|^{N+2s}}\qquad\textrm{and} \qquad K_0(x,y)= { c_{N,s}}\frac{1}{|x-y|^{N+2s}}.
\ee
Then \eqref{eq:def-diffeom} gives rise to the well known expansions
\be \label{eq:Jacob}
\textrm{Jac}_{\Phi_\e}(x)=1+\e\textrm{div}X(x)+ O(\e^2),\qquad \partial_\eps 
\textrm{Jac}_{\Phi_\e}(x) = \textrm{div}X(x)+ O(\eps)
\ee
uniformly in $x \in \R^N$, where $X:=\de_\e \big|_{\e=0} \Phi_\e  \in C^{0,1}(\R^N;\R^N)$ and therefore $\textrm{\div}X$ is a.e. defined on $\R^N$. From \eqref{eq:def-diffeom}, we also get
$$
{|\Phi_\e(x)-\Phi_\e(y)|^{-N-2s}}=|x-y|^{-N-2s}\left(  1+ 2\e\frac{x-y}{|x-y|} \cdot P_{X}(x,y)+O(\e^2)     \right)^{-\frac{N+2s}{2}},
$$
and 
$$
\partial_\e {|\Phi_\e(x)-\Phi_\e(y)|^{-N-2s}}=|x-y|^{-N-2s}\left(- (N+2s) \frac{x-y}{|x-y|} \cdot P_{X}(x,y)+O(\e)\right),
$$
uniformly in $x,y \in \R^N$, $x \not = y$ with 
$$
P_X \in L^\infty(\R^N \times \R^N), \qquad P_{X}(x,y)=\frac{X(x)-X(y)}{|x-y|}.
$$
Moreover  by \eqref{eq:Jacob} and  the fact that $\de_\e \Phi_\e$, $X\in C^{0,1}(\R^N)$,   we have  that 
\begin{align}
 K_\e(x,y)=K_0(x,y)+\e\de_\e\Big|_{\e=0}  K_\e(x,y)+O(\e^2) K_0(x,y) , \label{eq:estim-K-eps-1-1}
\end{align}
and 
\begin{align}
\partial_\e K_\e(x,y)=\de_\e\Big|_{\e=0}K_\e(x,y) +O(\e)K_0(x,y), \label{eq:estim-K-eps-1-2}
\end{align}
uniformly in $x,y \in \R^N$, $x \not = y$, where 
\begin{align}
\de_\e\Big|_{\e=0}  K_\e(x,y)=-  \Bigl[(N+2s)\frac{x-y}{|x-y|} \cdot P_{X}(x,y)-&(\textrm{div}X(x)+ \textrm{div}X(y))\Bigr] K_0(x,y).  \label{eq:estim-K-eps-2}
\end{align}
In particular, it follows from (\ref{eq:estim-K-eps-1-1}) and (\ref{eq:estim-K-eps-2}) that there exist $\e_0,C>0$ with the property that 
\be \label{eq:Elliptic-K-e}
\frac{1}{C}K_0(x,y)\leq K_\e(x,y)\leq C K_0(x,y) \qquad\textrm{ for all $x, y\in \R^N$, $x \not =y$ and $\e\in (-\e_0,\e_0)$.}
\ee
For $v\in \cH^s_0(\Omega)$ and $\e\in (-\eps_0,\eps_0)$, we now define   
\be\label{eq:def-j-Uk}
\cV_{v}(\e):= \frac{1}{2}\int_{\R^{2N}}  {(v(x)-v(y))^2}  K_\e(x,y)    dxdy . 
\ee
Then, by \eqref{eq:def-lambda-sp},   \eqref{eq:def-diffeom} and a change of variables,  we have the following variational characterization for $\lambda_{s,p}(\O_\e)$: 
\begin{align}\label{eq:var-carac-lOe}
 \lambda_{s,p}^\e &:= \lambda_{s,p}(\O_\e)=\inf \left\{[u]_s^2\::\:  u\in  \cH^s_0(\Omega_\eps), \quad\int_{\O_\eps} |u|^p\, dx=1  \right\}\nonumber\\
&=\inf \left\{ \cV_{v}(\e) \::\:  v\in  \cH^s_0(\Omega),\quad\int_{\O} |v|^p\textrm{Jac}_{\Phi_\e}(x)\, dx=1  \right\} \quad \text{for $\e \in (-\e_0,\e_0)$.}
\end{align}
As mentioned earlier, we prefer to use (\ref{eq:var-carac-lOe}) from now on where the underlying domain is fixed and the integral terms depend on $\eps$ instead.
It follows from (\ref{eq:estim-K-eps-1-1}) and (\ref{eq:estim-K-eps-1-2}) that, for given $v \in \cH^s_0(\Omega)$, the function $\cV_v: (-\eps_0,\eps_0) \to \R$ is of class $C^1$ with 
\begin{equation}
\label{cV-v-diff-property-zero}
\cV_v'(0) = \frac{1}{2}\int_{\R^{2N}}  {(v(x)-v(y))^2} \de_\e\big|_{\e=0} K_\e(x,y)    dxdy,
\end{equation}
where $\de_\e\big|_{\e=0} K_\e(x,y)$ is given in (\ref{eq:estim-K-eps-2}), 
\begin{equation}
\label{cV-v-deriv-zero-bounded}
|\cV_v'(0)| \le C [v]_s^2 \qquad \text{with a constant $C >0$}
\end{equation}
and we have the expansions 
\begin{equation}
\label{cV-v-diff-property}
\cV_v(\eps) = \cV_v(0)+ \eps \cV_v'(0) + O(\eps^2)[v]_s^2, \qquad
\cV_v'(\eps) = \cV_v'(0) + O(\eps)[v]_s^2 
\end{equation}
with $O(\eps)$, $O(\eps^2)$ independent of $v$. From \eqref{eq:Jacob}, \eqref{eq:Elliptic-K-e} and the variational characterization (\ref{eq:var-carac-lOe}), it is easy to see that 
$$
\frac{1}{C} \le \l_{s,p}^\e \leq C \qquad \text{for all $\e\in (-\e_0,\e_0)$ with some constant $C>0$.} 
$$
Using this and \eqref{eq:Jacob}, \eqref{eq:Elliptic-K-e} once more, we can show that 
\be\label{eq;bound-v-eps}
\frac{1}{C}\leq \|v_\e\|_{L^p(\O)}\leq C\qquad \text{and}\qquad  \frac{1}{C} \le [v_\e]_s\leq C .
\ee
for every $\e\in (-\e_0,\e_0)$ and every minimizer $v_\e \in \cH^s_0(\Omega)$ for (\ref{eq:var-carac-lOe}) with a constant $C>0$. 
%Following the arguments in  \cite{CM}, we may then also deduce that $\|v_\e\|_{ L^\infty(\R^N)}\leq C$. 

The following lemma is essentially a corollary of Lemma~\ref{lem:v_ez_k-ok-prelim}. 

\begin{lemma}
\label{lem:v_ez_k-ok-corollary}
Let $(v_k)_k$ be a sequence in $\cH^s_0(\Omega)$ with 
$v_k \to v$ in $\cH^s_0(\Omega)$. Then we have 
$$
\lim_{k\to\infty}\cV_{v_{k}}(0)= \cV_{v}(0)\qquad \text{and}\qquad  \lim_{k\to\infty}\cV_{v_k}'(0)= \cV_{v}'(0). 
$$
\end{lemma}

\begin{proof}
The first limit is trivial since $\cV_v(0)= [v]_s^2$ for $v \in \cH^s_0(\Omega)$. The second limit follows from Lemma~\ref{lem:v_ez_k-ok-prelim}, (\ref{eq:estim-K-eps-2}) and (\ref{cV-v-diff-property-zero}) by noting that $\mu \in L^\infty(\R^N \times \R^N)$ for the function 
$$
\mu(x,y)=- (N+2s) \frac{x-y}{|x-y|} \cdot P_{X}(x,y) + (\textrm{div}X(x)+ \textrm{div}X(y)).
$$
\end{proof}

\section{One-sided Shape derivative computations}\label{S:Onse-side}
We keep using the notation of the previous sections, and we recall in particular the variational characterization of $\l_{s,p}^\e=\l_{s,p}(\O_\e)$ given in (\ref{eq:var-carac-lOe}). 
The aim of this section is to prove the following  result.
\begin{proposition}\label{prop:shape-deriv}
We have 
\begin{align*}
 \de_\e^+ \Big|_{\e=0}\l_{s,p}^\e
=\min &\left\{ 2  \k_s  \int_{\de\O} (u/\d^s)^2X\cdot \nu \, dx\::\, u\in \calH  \right\},
\end{align*}
where    $\calH$ is the set of positive minimizers for $\l_{s,p}^0:= \l_{s,p}(\O)$, $X:=\de_\e \big|_{\e=0} \Phi_\eps$ and $\k_s$ is given by \eqref{eq:-kappa-s-implicit}. 
\end{proposition}

The proof of Proposition \ref{prop:shape-deriv} requires several preliminary results. We start with a formula for the derivative of the function given by  \eqref{eq:def-j-Uk}.

\begin{lemma}\label{lem:j1kprimes-preliminary}
Let $U \in C^{1,1}_c(\O)$. Then
\begin{align}
%\label{eq:rho1-prime-prelim}
&\cV_{U}'(0)= -2 \int_{\R^N} \nabla U \cdot X (-\Delta)^s U dx.
\end{align}
\end{lemma}

\begin{proof}
By \eqref{eq:estim-K-eps-2}, (\ref{cV-v-diff-property}) and Fubini's theorem, we have 
\begin{align*}
\cV_{U}'(0)=& \frac{-(N+2s) c_{N,s}}{2}\int_{\R^{2N}}  {(U(x)-U(y))^2}\frac{(x-y)\cdot(X(x)-X(y))}{|x-y|^{N+2s+2}}   dx dy \nonumber\\
&+ \frac{1}{2}\int_{\R^{2N}}  {(U(x)-U(y))^2}  K_0(x,y)  (\textrm{div}X(x)+ \textrm{div}X(y))dxdy\\
=& \frac{-(N+2s) c_{N,s}}{2} \lim_{\mu \to 0} \int_{|x-y|>\mu}  {(U(x)-U(y))^2}\frac{(x-y)\cdot(X(x)-X(y))}{|x-y|^{N+2s+2}}   dxdy \nonumber\\
&+ \int_{\R^{2N}}  {(U(x)-U(y))^2}  K_0(x,y)\textrm{div}X(x) dxdy\\
=&-(N+2s) c_{N,s} \lim_{\mu \to 0} \int_{\R^N} \int_{\R^N \setminus \overline{B_\mu(y)}} {(U(x)-U(y))^2}\frac{(x-y)\cdot X(x)}{|x-y|^{N+2s+2}}   dxdy \nonumber\\
&+ \int_{\R^{2N}}  {(U(x)-U(y))^2}  K_0(x,y)\textrm{div}X(x) dxdy
\end{align*}
{\cred Applying, for fixed $y \in \R^N$ and $\mu>0$, the divergence theorem in the domain $\{x \in \R^{N}\::\: |x-y| > \mu \}$ and using that $\n_{x} |x-y|^{-N-2s}=-(N+2s)\frac{x-y}{|x-y|^{N+2s+2}}$,} we obtain
\begin{align}
\cV_{U}'(0) =&c_{N,s} \lim_{\mu\to 0} \int_{\R^N} \int_{\R^N \setminus \overline{B_\mu(y)}}  {(U(x)-U(y))^2} \n_{x} |x-y|^{-N-2s} \cdot X(x)  dxdy \nonumber\\
&+ \int_{\R^{2N}}  {(U(x)-U(y))^2}  K_0(x,y)\textrm{div}X(x) dxdy \nonumber\\
=&- \lim_{\mu\to 0} \int_{\R^N} \int_{\R^N \setminus \overline{B_\mu(y)}}  {(U(x)-U(y))^2}  K_0(x,y) \textrm{div}X(x)dxdy   \nonumber\\
& -  \lim_{\mu\to 0}\int_{\R^N} \int_{\R^N \setminus \overline{B_\mu(y)}} {(U(x)-U(y))\n U(x) \cdot X(x) }  K_0(x,y)   dxdy  \nonumber\\
&+   \lim_{\mu\to 0}  \int_{\R^N} \int_{\partial B_\mu(y)}  {(U(x)-U(y))^2}   \frac{y-x}{|x-y|}\cdot X(x)  K_0(x,y)       \, d\s(y)\,dx  \nonumber\\
&+ \int_{\R^{2N}}  {(U(x)-U(y))^2}  K_0(x,y)  \textrm{div}X(x)dxdy\nonumber\\
=& -  \lim_{\mu\to 0}\int_{|x-y|>\mu} (U(x)-U(y))\n U(x) \cdot X(x)   K_0(x,y)   d(x,y)  \nonumber\\
&+   \lim_{\mu\to 0}\mu^{-N-1-2s}  \int_{|x-y|=\mu}  {(U(x)-U(y))^2} (y-x) \cdot X(x)    \, d\s(x,y)\nonumber\\
=&-  \frac{c_{N,s}}{2} \lim_{\mu\to 0}\int_{\R^N}\n U(x) \cdot X(x)  \int_{\R^N \setminus \overline{B_\mu(0)}} \frac{2 U(x)-U(x+z)-U(x-z)}{|z|^{N+2s}}   dz dx  \nonumber\\
&+  \frac{1}{2} \lim_{\mu\to 0}\mu^{-N-1-2s}  \int_{|x-y|= \mu}  {(U(x)-U(y))^2} (y-x) \cdot (X(x)-X(y))    \, d\s(x,y) \label{revision-eq-1}
\end{align}
Since $U \in C^{1,1}_c(\Omega)$,  we have that
\begin{align}
  &\frac{c_{N,s}}{2} \lim_{\mu\to 0}\int_{\R^N} \n U(x) \cdot X(x) \int_{\R^N \setminus \overline{B_\mu(0)}} \frac{2 U(x)-U(x+z)-U(x-z)}{|z|^{N+2s}}   dz dx\nonumber \\
  &= \frac{c_{N,s}}{2} \int_{\R^N}\n U(x) \cdot X(x)  \int_{\R^N} \frac{2 U(x)-U(x+z)-U(x-z)}{|z|^{N+2s}}   dz dx\nonumber\\
&= \int_{\R^N} (-\Delta)^sU(x)  \n U(x) \cdot X(x) dx. \label{revision-eq-2}
\end{align}
Moreover, since $U$ is compactly supported, we may fix $R>0$ large enough such that
$(U(x)-U(y))^2 = 0$ for all $x,y \in B_R(0)$ with $|x-y|<1$.
Setting $N_\mu:= \{(x,y) \in B_R(0) \times B_R(0) \::\: |x-y|= \mu\}$ for $0<\mu<1$ and using that
$U, X \in C^{0,1}(\R^N)$, we thus deduce that 
\begin{align}
  &\mu^{-N-1-2s} \int_{|x-y|= \mu}  {(U(x)-U(y))^2} (y-x) \cdot (X(x)-X(y))    \, d\s(x,y)\nonumber\\
  &= \mu^{-N-1-2s} \int_{N_\mu}  {(U(x)-U(y))^2} (y-x) \cdot (X(x)-X(y))\, d\s(x,y) = O(\mu^{3-1-2s}) \to 0,\label{revision-eq-3}
\end{align}
as $\mu \to 0$, since the $2N-1$-dimensional measure of the set $N_\mu$ is of order $O(N-1)$ as $\mu \to 0$.
The claim now follows by combining \eqref{revision-eq-1},\eqref{revision-eq-2} and \eqref{revision-eq-3}.
\end{proof}

We cannot apply Lemma~\ref{lem:j1kprimes-preliminary} directly to minimizers $u \in \cH^s_0(\Omega)$ of $\l_{s,p}(\O)$ since these are not contained in $C^{1,1}_c(\Omega)$. The aim is therefore to apply Lemma~\ref{lem:j1kprimes-preliminary} to $U_k:=u \z_k \in C^{1,1}_c(\Omega)$ with $\z_k$ given in (\ref{eq:def-z-k-etc}), and to use Proposition~\ref{lem:lim-ok}. This leads to the following derivative formula which plays a key role in the proof of Proposition~\ref{prop:shape-deriv}.
\begin{lemma}\label{lem:j1kprimes}
Let $u\in \cH^s_0(\Omega)$ be a solution to \eqref{eq:1.1}.   Then we have 
$$
\cV_u'(0)=
\frac{2\l_{s,p}(\O) }{p}\int_{\O} u^p \div \,X \,dx + 2  \k_s  \int_{\de\O} (u/\d^s)^2 X\cdot \nu\, dx.
$$   
\end{lemma}

\begin{proof}
By Lemma~\ref{reg-prop-minimizers} and since $\Omega$ is of class $C^{1,1}$, we have $U_k:=u \zeta_k\in C^{1,1}_c(\Omega) \subset \cH^s_0(\Omega)$ for $k \in \N$, and $U_k \to u$ in $\cH^s_0(\Omega)$ by Lemma~\ref{inner-approx-property}. Consequently, $\cV_u'(0)= \lim \limits_{k \to \infty}\cV_{U_k}'(0)$ by Corollary~\ref{lem:v_ez_k-ok-corollary}, so it remains to show that 
\begin{equation}
  \label{eq:remains-to-show}
\lim_{k \to \infty}\cV_{U_k}'(0) =  \frac{2\l_{s,p}(\O) }{p}\int_{\O}   u^p \div X \,dx  + 2  \k_s  \int_{\de\O} (u/\d^s)^2 X\cdot \nu\, dx.
\end{equation}
Applying Lemma~\ref{lem:j1kprimes-preliminary} to $U_k$, we find that 
$$
\cV_{U_k}'(0)= -2 \int_{\R^N} \nabla U_k \cdot X (-\Delta)^s U_k dx  \qquad \text{for $k \in \N$.}
$$
By the standard product rule for the fractional Laplacian, we have $\Ds U_k=u\Ds \z_k + \z_k \Ds u-I(u, \z_k)$ with $I(u, \z_k)$ given by \eqref{eq:def-I-u-v}. We thus obtain
\begin{align}
&\cV_{U_k}'(0)=- 2\int_{\R^N}  \n U_k \cdot X  \z_k  \Ds u\, dx-2\int_{\R^N}  [\n U_k \cdot X]  u \Ds \z_k \, dx \label{eq:first-est-jk1-pr}
\\
&\qquad\qquad+  2 \int_{\R^N}  \n U_k \cdot X I(u, \z_k)\, dx \nonumber\\
&=-2\l_{s,p}(\O) \int_{\O}  \n U_k \cdot X  \z_k u^{p-1}\, dx -2 \int_{\R^N}\n U_k \cdot X \Bigl( u   \Ds \z_k  - I(u, \z_k)\Bigr)\,dx,\nonumber
\end{align}
where we used that $\Ds u=\l_{s,p}(\O) u^{p-1}$ in $\O$. Consequently, Proposition~\ref{lem:lim-ok} yields that 
\begin{equation}
\label{limit-formula-prelim}
\lim_{k \to \infty}\cV_{U_k}'(0)= -2\l_{s,p}(\O) \lim_{k \to \infty}  \int_{\O}  \n U_k\cdot X \z_k u^{p-1}\, dx + 2 \k_s  \int_{\de\O} \psi^2 X\cdot \nu\, dx.
\end{equation}
Moreover, integrating by parts, we obtain, for $k \in \N$, 
\begin{align}
\int_{\O}  [\n U_k\cdot &X] \z_k u^{p-1}\, dx = \frac{1}{p} \int_{\O}  [\n u^{p} \cdot X] \z_k^2  \, dx+ \int_{\O}  [\n 
\z_k \cdot X] \z_k u^{p}\, dx \nonumber \\
&=-\frac{1}{p}\int_{\O}   u^p \div X \z_k^2 \,dx -\frac{2}{p}\int_{\O}   u^p \z_k [X\cdot\n \z_k] \,dx + \int_{\O} u^{p} \z_k  [X\cdot \n \z_k]  \, dx.\label{eq:rho1-prime}
\end{align}
Since $u^p\in C^s_0(\ov\O)$ by Lemma~\ref{reg-prop-minimizers}, it is easy to see from the definition of $\z_k$ that the last two terms in \eqref{eq:rho1-prime} tend to zero as $k\to \infty$, whereas
$$
\lim_{k\to\infty}\int_{\O}   u^p\textrm{div} X \z_k^2 \,dx= \int_{\O}   u^p\textrm{div}  X \,dx. 
$$
Hence 
$$
\lim_{k \to \infty} \int_{\O}  \n U_k\cdot X \z_k u^{p-1} \, dx
= - \frac{1}{p} \int_{\O}   u^p\textrm{div}  X \,dx. 
$$
Plugging this into (\ref{limit-formula-prelim}), we obtain (\ref{eq:remains-to-show}), as required. 
\end{proof}

Our next lemma provides an upper estimate for $ \de_\e^+ \Big|_{\e=0}\l_{s,p}^\e$.

\begin{lemma}\label{lem:upper-estim} Let $u\in \cH$ be a positive minimizer for $\lambda_{s,p}^0 = \l_{s,p}(\O)$. Then 
\begin{align}\label{eq:upper-bnd}
\limsup_{\eps \to 0^+}\frac{\l_{s,p}^\e-\l_{s,p}^0}{\e} \leq& 2  \k_s  \int_{\de\O} (u/\d^s)^2 X\cdot \nu\, dx .
\end{align}
\end{lemma}
\begin{proof}
For $\eps \in (-\eps_0,\eps_0)$, we define 
$$
j(\eps):=\frac{\cV_{u}(\e)}{\tau(\e)}\quad \text{for $k \in \N$ with} 
\quad \tau(\e):=  \left( \int_{\O}|u|^p  \textrm{Jac}_{\Phi_\e}(x)\, dx \right)^{2/p}.
$$
By \eqref{eq:var-carac-lOe}, we then have $\l_{s,p}^\e \leq j(\e)$ for $\eps \in (-\eps_0,\eps_0)$. Moreover, 
$$
\tau(0)=\|u\|_{L^p(\Omega)}^{2/p}=1, \quad \cV_{u}(0)= [u]_s^2 = \l_{s,p}(\O)\quad \text{and}\quad j(0)= \frac{\cV_{u}(0)}{\tau(0)} = \l_{s,p}^0,
$$
which implies that  
$$
\de_\e^+ \Big|_{\e=0}\l_{s,p}^\e\leq j'(0) = 2  \k_s  \int_{\de\O} (u/\d^s)^2 X\cdot \nu \, dx,
$$ by Lemma \ref{lem:j1kprimes} and \eqref{eq:Jacob}, as claimed.
\end{proof}

Next, we shall prove a lower estimate for $ \de_\e^+ \Big|_{\e=0}\l_{s,p}^\e$.

\begin{lemma}\label{lem:lower-estim}
We have 
$$
\liminf_{\e\searrow 0^+}\frac{\l_{s,p}^\e-\l_{s,p}^0}{\e}\geq \inf\left\{ 2\k_s \int_{\de\O}(u/\d^s)^2 X\cdot \nu\,dx\::\,  u\in \cH \right\}.
$$
\end{lemma}

\begin{proof}
Let $(\e_n)_n$ be a sequence  of positive numbers converging to zero and with the property that 
\be \label{eq:lim-to-lim-inf}
\lim_{n\to \infty}\frac{  \l_{s,p}^{\e_n}-  \l_{s,p}^0}{\e_n} = \liminf_{\e\searrow 0^+}\frac{\l_{s,p}^\e-\l_{s,p}^0}{\e}.
\ee
For  $n\in\N$, we let $v_{\e_n}$ be a positive minimizer corresponding to the variational characterization of $\l_{s,p}^{\e_n}$ given in (\ref{eq:var-carac-lOe}), i.e. we have 
\begin{equation}
  \label{eq:e-n-prime-eq}
\l_{s,p}^{\e_n}  = \cV_{v_{\e_n}}(\e_n) \qquad \text{and}\qquad 
\int_{\O} v_{\e_n}^p\textrm{Jac}_{\Phi_{\e_n}}dx = 1.
\end{equation}
Since $v_{\e_n}$ remains bounded in $\cH^s_0(\Omega)$ by \eqref{eq;bound-v-eps}, we may pass to a sub-sequence with the property that $v_{\e_n} \weak u$ in $\cH^s_0(\Omega)$ for some $u\in \cH^s_0(\Omega)$. Moreover, $v_{\e_n} \to u$ in $L^p(\Omega)$ as $n \to \infty$ since the embedding $\cH^s_0(\Omega) \to L^p(\Omega)$ is compact. 
In the following, to keep the notation simple, we write $\e$ in place of $\e_n$. 
By (\ref{cV-v-deriv-zero-bounded}), (\ref{cV-v-diff-property}) and (\ref{eq:e-n-prime-eq}), we have  
\begin{equation}\label{eq:ATl-prelim}
\cV_{v_{\e}}(0)= \cV_{v_\e}(\e)- \eps \cV_{v_\eps}'(0) + O(\eps^2)[v_\eps]_s^2 = \l_{s,p}^\e - \eps \cV_{v_\eps}'(0) + O(\eps^2)=\l_{s,p}^\e + O(\eps)
\end{equation}
and therefore 
\begin{equation}
\label{strong-convergence-v-eps-prelim}
\cV_{u}(0) =[u]_s^2 \le  \liminf_{\eps \to 0} [v_\eps]_s^2
=  \liminf_{\eps \to 0}\cV_{v_{\e}}(0) \le \limsup_{\eps \to 0}\l_{s,p}^\e \le \l_{s,p}^0,
\end{equation}
where the last inequality follows from Lemma~\ref{lem:upper-estim}. In view of \eqref{eq:Jacob} and the strong convergence $v_\eps \to u$ in $L^p(\Omega)$, we see that
\begin{equation}
  \label{eq:div-new-expansion-v-eps}
1 = \int_{\O} v_\e^p\textrm{Jac}_{\Phi_\e}dx= \int_{\O} v_\e^p(1+\e \textrm{div}X )dx+O(\e^2) = \int_{\Omega} u^p\,dx +o(1)
\end{equation}
as $\eps \to 0$, and hence $\|u\|_{L^p(\Omega)}=1$. Combining this with \eqref{strong-convergence-v-eps-prelim}, we see that $u \in \cH$ is a minimizer for $\l_{s,p}^0$, and that equality must hold in all inequalities of \eqref{strong-convergence-v-eps-prelim}. From this we deduce that 
\begin{equation}
\label{strong-convergence-v-eps}
\text{$v_\eps \to u$ strongly in $\cH^s_0(\Omega).$}
\end{equation}
Now (\ref{eq:ATl-prelim}) and the variational characterization of $\l_{s,p}^0$ imply that  
\begin{equation}\label{eq:ATl}
\l_{s,p}^0 \left(\int_{\O} v_{\e}^pdx \right)^{2/p}\leq
\cV_{v_{\e}}(0)= \l_{s,p}(\O_\e) - \eps \cV_{v_\eps}'(0) + O(\eps^2)
\end{equation}
whereas by (\ref{eq:div-new-expansion-v-eps}) we have 
$$
\int_{\O} v_\e^p\,dx = 1-\e  \int_{\O} v_\e^p \textrm{div}X dx+O(\e^2) =
1-\e  \int_{\O} u^p \textrm{div}X dx+o(\e) 
$$ 
and therefore 
\begin{equation}
  \label{eq:div-expansion-lower-bound}
\left(\int_{\O} v_{\e}^pdx \right)^{2/p}=1- \frac{2 \e}{p}\int_{\O} u^p \textrm{div}X dx +o(\e).
\end{equation}
Plugging this into (\ref{eq:ATl}), we get the inequality
$$
\l_{s,p}^\e \ge \Bigl(1- \frac{2 \e}{p}\int_{\O} u^p \textrm{div}X dx\Bigr)
\l_{s,p}^0 + \eps \cV_{v_\eps}'(0)+o(\eps).
$$
Since, moreover, $\cV_{v_\eps}'(0) \to \cV_u'(0)$ as $\eps \to 0$ by Lemma~\ref{lem:v_ez_k-ok-corollary} and (\ref{strong-convergence-v-eps}), it follows that 
$$
\l_{s,p}^\eps-\l_{s,p}^0 \ge \eps \Bigl(\cV_{u}'(0) -\frac{2 \l_{s,p}^0}{p}\int_{\O} u^p \textrm{div}X dx\Bigr)+o(\eps)
$$
and therefore 
$$
\l_{s,p}^\e-\l_{s,p}^0 \ge 2 \eps \k_s \int_{\de\O}(u/\d^s)^2 X\cdot \nu\,dx+o(\e)
$$
by Lemma~\ref{lem:j1kprimes}. We thus conclude that 
$$
\lim_{\e \to 0^+}\frac{\l_{s,p}^\e-\l_{s,p}^0}{\e}
\ge 2\k_s \int_{\de\O}(u/\d^s)^2 X\cdot \nu\,dx.
$$
Taking the infinimum over $u\in \cH$ in the RHS of this inequality and using \eqref{eq:lim-to-lim-inf}, we get the result.
\end{proof}

\begin{proof}[Proof of Proposition \ref{prop:shape-deriv} (completed)]
Proposition \ref{prop:shape-deriv} is a consequence of Lemma \ref{lem:upper-estim} and Lemma \ref{lem:lower-estim}. Indeed,  let 
$$
A_{s,p}(\O):= \inf \left\{2\k_s \int_{\de\O}(u/\d^s)^2 X\cdot \nu\,dx \::\,  u\in \cH    \right\}.
$$
Thanks to  \eqref{eq:Bnd-reg-v-eps}  the infinimum  $A_{s,p}(\O)$ is attained.
Finally by Lemma \ref{lem:upper-estim} and Lemma \ref{lem:lower-estim} we get
\begin{align*}
A_{s,p}(\O) \geq   \de_\e^+ \Big|_{\e=0}\l_{s,p}^\e \geq   \liminf_{\e\searrow 0}\frac{\l_{s,p}^\e-\l_{s,p}^0}{\e}\geq A_{s,p}(\O).
\end{align*}
\end{proof}

\section{Proof of the main results}\label{S:proofs}
 
In this section we complete the proofs of the main results stated in the introduction. 

\begin{proof}[Proof of Theorem \ref{th:Shape-deriv} (completed)]
%The following lemma provides the value of the  limits in  Proposition \ref{prop:shape-deriv}.
In view of Proposition \ref{prop:shape-deriv}, the proof of  Theorem \ref{th:Shape-deriv} is complete once we show that 
\be\label{eq:k-s-Gamma}
 2\k_s=\G(1+s)^2,
\ee 
where $\Gamma$ is the usual Gamma function.
In view of \eqref{eq:-kappa-s-implicit},  the constant $\k_s$ does not depend on $N$, $p$ and $\O$,  we consider the case $N=p=1$ and the family of diffeomorphisms $\Phi_\e$ on $\R^N$ given by $\Phi_\e(x)=(1+\e)x$, $\e \in (-1,1)$, so that $X:= \partial_\e \big|_{\e = 0} \Phi_\e$ is simply given by $X(x)=x$. Letting $\O_0:=(-1,1)$, we define $\O_{\e}=\Phi_\e(\O_0)=(-1-\e,1+\e)$. Moreover, we consider $w_\e\in \cH^s_0(\Omega_\e) \cap C^s_0([-1-\e,1+\e])$ given by 
\be \label{eq:express-ells}
w_\e(x)=\ell_s((1+\e)^2-|x|^2)^s_+ \qquad\textrm{ with }\qquad \ell_s:= \frac{2^{-2s}\G(1/2)}{\G(s+1/2)\G(1+s)}.
\ee
It is well known that $w_\eps$ is the unique solution of the problem 
$$
\Ds w_\e=1 \quad \text{in $\O_\e$,}\qquad w_\e \equiv 0 \quad \text{on $\R^N \setminus \Omega_\e$,}
$$
see e.g. \cite{RX-Poh} or \cite{JW17}. Recalling (\ref{eq:1.1}), we thus deduce that $u_\e = \l_{s,1}(\O_\e)w_\e$ is the unique positive minimizer corresponding to \eqref{eq:def-lambda-sp} in the case $N=p=1$, which implies that $\|u_\e\|_{L^1(\R)}=1$ and therefore 
\begin{equation}
  \label{eq:-to-diff-eq}
\l_{s,1}(\O_\e)= \|w_\e\|_{L^1(\R)}^{-1}= (1+\e)^{-(2s+1)} \|w_0\|_{L^1(\R)}^{-1}.
\end{equation}
Moreover, by standard properties of the Gamma function,
\begin{align*}
\|w_0\|_{L^1(\R)} &=\ell_s \int_{-1}^1 (1-|x|^2)^s\, dx=2 \ell_s \int_0^1(1-r^2)^s\, dr= \ell_s \int_0^1 t^{-1/2} (1-t)^s\, dt\\
&= \ell_s \frac{\G(1/2)\G(s+1)}{\G(s+3/2)}= \ell_s \frac{\G(1/2)\G(s+1)}{(s+1/2)\G(s+1/2)}= \frac{2^{2s}\,\ell_s^2\,\G(s+1)^2}{s+1/2}.
\end{align*}
By differentiating (\ref{eq:-to-diff-eq}), we get
\be \label{eq:diff-simple-case}
\de_\e\Big|_{\e=0}\l_{s,1}(\O_\e)  =-\frac{2s+1}{\|w_0\|_{L^1(\R)}}.
\ee
On the other hand, by Proposition \ref{prop:shape-deriv} and the fact that $u_0$ is the unique positive minimizer for $\lambda_{s,1}$,  we deduce that 
$$
\de_\e^+\Big|_{\e=0}\l_{s,1}(\O_\e)= -2\k_s [(u_0/\delta^s)^2(1)+(u_0/\delta^s)^2(-1)]= -2^{2+2s} \k_s \,\ell_s^2 \,\lambda_{s,1}(\O_0)^2= -\frac{2^{2+2s} \k_s\, \ell_s^2}{\|w_0\|_{L^1(\R)}^2}.
$$
We thus conclude that 
$$
2 \k_s = \frac{(2s+1)\|w_0\|_{L^1(\R)}}{2^{1+2s} \ell_s^2} = \Gamma(s+1)^2.
$$
Thus, by Proposition \ref{prop:shape-deriv}, we get the result as stated in the theorem.
\end{proof}

\noindent
%In view of  Lemma \ref{lem:lim-ok},  we have that 
% The proof of Theorem \ref{th:Shape-deriv} is completed by the following lemma.

%We will finish the proof of this lemma in the appendix. Precisely, we will show that $2\k_s = -\G^2(1+s)$.
%\begin{proof}[Proof of Corollary \ref{cor:1.2}]
%This is an immediate consequence of Theorem\ref{th:Shape-deriv} and the simplicity of $\l_{s,p}(\O)$ which implies that $\cH$ contains a unique element.
 %\end{proof}
 
 \begin{proof}[Proof of Corollary \ref{cor:1.3}]
Let     $h\in C^{3}(\de\O)$, with $\int_{\de\O}h\, dx=0$. Then it is well known (see e.g. \cite[Lemma 2.2]{FW18}) that there exists a family of diffeomorphisms $\Phi_\eps: \R^N \to \R^N$, $\eps \in (-1,1)$ satisfying \eqref{eq:def-diffeom} and having the following properties:   
\be
\text{$|\Phi_\eps(\Omega)|= |\Omega|$ for $\e \in (-1,1)$, and $X:= \partial_\e \big|_{\e=0}\Phi_\e$ equals $h \nu$ on $\partial \Omega$.}
\label{eq:Vectorfil-Y}
\ee
By assumption, there exists $\eps_0 \in (0,1)$ with $\l_{s,p}(\Phi_\e(\O)) \geq \l_{s,p}(\O)$ for $\e \in (-\e_0,\e_0)$.  
Applying Theorem \ref{th:Shape-deriv}  and noting that $X \cdot \nu \equiv h$ on $\partial \Omega$ by \eqref{eq:Vectorfil-Y}, we get
$$
 \min \left\{\G(1+s)^2 \int_{\de\O}(u/\d^s)^2 h\,dx\::\, u\in \calH    \right\}\geq 0.
$$
By the same argument applied to $-h$, we get 
 \begin{equation}
\max \left\{\G(1+s)^2 \int_{\de\O}(u/\d^s)^2 h\,dx\::\, u\in \calH    \right\}\leq 0.
\end{equation} 
We thus conclude that
 $$
\int_{\de\O}(u/\d^s)^2 h\,dx = 0 \qquad \text{ for every $u\in \calH$ and  for all  $h\in C^3(\de\O)$, with $\int_{\de\O}h\, dx=0$.}
$$
By a standard argument, this implies that $u/\d^s$ is constant on $\de\O$. Now, since $u$ solves \eqref{eq:1.1} and $p\in \{1\}\cup [2,\infty)$, we deduce from \cite[Theorem 1.2]{JW17} that $\O$ is a ball.
\end{proof}
\begin{proof}[Proof of Theorem \ref{th:1.2}]
%\begin{proof}
Consider the unit centered ball $B_1=B_1(0)$.  For $\tau\in (0,1)$ and $t\in (\tau-1,1-\tau)$, we define $B^t:=B_{\tau}(te_1)$, where $e_1$ is the first coordinate direction.
To prove Theorem \ref{th:1.2}, we can take advantage  of the   invariance   under rotations of the problem and may restrict our attention to domains of the form $\Omega(t) = B_1\setminus \overline{B^t} $.   We define 
\begin{equation}
  \label{eq:def-theta}
\theta:(\tau-1,1-\tau)\to \R, \qquad \theta(t):=\l_{s,p}(\O(t)).
\end{equation}
We claim that $\theta$ is differentiable and satisfies 
 \begin{equation}\label{4.4}
\theta'(t)<0  \qquad\textrm{ for $t\in (0,1-\tau)$.}
\end{equation} 
For this we fix $t \in (\tau-1,1-\tau)$ and a vector field $X: \R^N \to \R^N$ given by $X(x) = \rho(x)e_{1}$, where $\rho \in C^{\infty}_{c}(B_{1})$ satisfies $\rho \equiv 1$ in a neighborhood of $B^t$. For $\eps \in (-1,1)$, we then define 
$\Phi_\eps: \R^N \to \R^N$ by $\Phi_\eps(x)=x + \beta \eps X(x)$, where $\beta>0$ is chosen sufficiently small to guarantee that 
$\Phi_\eps$, $\eps \in (-1,1)$ is a family of diffeomorphisms satisfying \eqref{eq:def-diffeom} and satisfying $\Phi_\eps(B_1) = B_1$ for $\eps \in (-1,1)$. Then, by construction, we have  
\begin{equation}
  \label{eq:o-t-characterization}
\Phi_\eps(\O(t))= \Phi_{\eps}\left(B_1\setminus \overline{B^t}\right)  = 
B_1\setminus \overline{\Phi_{\eps}(B^t)}= 
B_1\setminus \overline{B^{t+\beta \eps}} = \O(t+\beta \eps).
\end{equation}
Next we recall that, since $p \in \{1,2\}$, there exists a unique positive minimizer $u \in \cH^s_0(\Omega(t))$ corresponding to the variational characterization \eqref{eq:def-lambda-sp} of $\lambda_{s,p}(\O(t))$. Hence, by Corollary~\ref{cor:1.2}, the map $\eps \mapsto \l_{s,p}(\Phi_\eps(\O(t)))$ is differentiable at $\eps = 0$. In view of (\ref{eq:o-t-characterization}), we thus find that the map $\theta$ in (\ref{eq:def-theta}) is differentiable at $t$, and 
\begin{equation}
  \label{eq:theta-prime-formula}
\theta'(t)= \frac{1}{\beta} \frac{d}{d\eps} \Big|_{\eps=0}\l_{s,p}(\Phi_\eps(\O(t))) = \Gamma(1+s)^2\int_{\partial\Omega(t)}\left(\frac{u}{\d^{s}}\right)^{2} X\cdot \nu\, dx = \Gamma(1+s)^2\int_{\partial B^t}\left(\frac{u}{\d^{s}}\right)^{2} \nu_1\, dx
\end{equation}
by (\ref{eq:cor:1.2-formula}). Here $\nu$ denotes the interior unit normal on $\partial \Omega(t)$ which coincides with the exterior unit normal to $B^t$ on $\partial B^t$, and we used that 
$$
X \equiv e_1\quad \text{on $\partial B^t$},\qquad X \equiv 0 \quad \text{on $\partial B_1 = \partial \Omega(t) \setminus \partial B^t$}
$$
to get the last equality in (\ref{eq:theta-prime-formula}). 
Next, for fixed $t\in (0,1-\tau)$, let $H$ be the half space defined by $H= \{x\in \R^N: x\cdot e_1> t\}$ and let $\Theta = H \cap \Omega(t)$. We also let  $r_H:\R^N\to \R^N$ be the reflection map with respect to he hyperplane $\partial H:=\{x\in \R^N: x\cdot e_1= t\}$.
For  $x\in \R^N$, we  denote $\bar x:= r_{H}(x)$,  $\overline u (x):= u(\overline x)$.  Using these notations,  we have 
\begin{align}
&\theta'(t)  = \Gamma(1+s)^2\int_{\partial B^{t}}\left(\frac{u}{\d^{s}}\right)^{2}\nu_{1}\,dx \nonumber\\
&=\Gamma(1+s)^2\int_{\partial B^{t}\cap\Theta}\left( \left(\frac{ u}{\d^{s}}\right)^{2}(x)-\left(\frac{\ov u}{\d^{s}} \right)^{2}( x)\right) \nu_{1}\, dx.  \label{eq:firslet-derap}
\end{align}
Let $w = \ov u -u \in H^s(\R^N)$.  Then $w$ is a (weak) solution of the problem 
\begin{equation}
 %\left\{\begin{aligned}
  \Ds w= \l_{s,p}(\Omega({t})) \ov  u^{p-1}- \l_{s,p}(\Omega({t}))  u^{p-1}= c_p w \qquad\text{in}\quad \Theta,
  %
%w\geq0&,\quad\text{on}\quad H^{+}\setminus\Theta\\
%
% \end{aligned}
%\right.
\end{equation}
where 
 $$
 \begin{cases}
 c_p:= \l_{s,p}(\Omega({t}))&\qquad \textrm { for $p=2$,}\\
 c_p=0 &\qquad \textrm { for $p=1$}.
 \end{cases} 
$$
Moreover, by definition, $w \equiv \ov u\geq 0$ in $H\setminus \ov\Theta$, and $w \equiv \ov u >0$ in the subset $[r_H(B_1) \cap 
H] \setminus \ov\Theta$ which has positive measure since $t>0$. Using that $w$ is anti-symmetric with respect to $H$ and the fact that $\l_{s,p}(\Theta)> c_p $ (which follows since $\Theta$ is a proper subdomain of $\O(t)$),  we can apply the weak maximum principle for antisymmetric functions  (see \cite[Proposition 3.1]{JW17} or \cite[Proposition 3.5]{JW}) to deduce that $w\ge 0$ in $\Theta$. Moreover, since $w \not \equiv 0$ in $\R^N$, it follows from the strong maximum principle for antisymmetric functions given in \cite[Proposition 3.6]{JW} that $w>0$ in $\Theta$. Now by  the fractional Hopf lemma for antisymmetric functions (see \cite[Proposition 3.3]{JW17}) we conclude that
$$
0< \frac{w}{\d^s}= \frac{\ov u}{\d^s}- {\frac{u}{\d^s}}\quad \text{and therefore}\quad {\frac{\ov u}{\d^s}} > \frac{u}{\d^s}  \geq 0     \qquad\textrm {on  $\de B^{t} \cap  \Theta$}.
$$
From this and \eqref{eq:firslet-derap} we get  \eqref{4.4}, since $\nu_1>0$ on $ \de B^{t}\cap  \Theta$. \\
To conclude, we observe that the function $t\mapsto\l_{s,p}(t)=\l_{s,p}(\O(t))$ is even,  thanks to the invariance of the problem under rotations. Therefore the function $\theta$ attains its maximum uniquely at $t=0$.
\end{proof}

\section{Proof of Proposition \ref{lem:lim-ok}}\label{s:proof-Lemma-conv}
The aim of this section is to prove Proposition \ref{lem:lim-ok}. For  the readers convenience,  we repeat the statement here.

\begin{proposition}\label{lem:lim-ok-new-section}
Let $X \in C^{0}(\overline \Omega,\R^N)$, let $u \in C^s_0(\ov \O)\cap C^1(\O)$, and assume that 
$\psi:=\frac{u}{\d^s}$ extends to a function on $\overline \Omega$ satisfying \eqref{eq:Bnd-reg-v-eps} and \eqref{eq:Grad-bnd-psi-eps}. Moreover, put $U_k:=u \z_k \in C^{1,1}_c(\Omega)$, where $\z_k$ is defined in (\ref{eq:def-z-k-etc}).  Then 
\begin{equation}
  \label{eq:lem:new-section-claim}
\lim_{k\to \infty}\int_{\Omega}\n U_k \cdot X \Bigl( u   \Ds \z_k  -I(u,\z_k)\Bigr)\,dx= - \k_s  \int_{\de\O} \psi^2 X\cdot \nu\, dx,
\end{equation}
where \be \label{eq:-kappa-s-implicit-new-section}
 \k_s:=  -\int_{\R} h'(r) \Ds h(r)\, dr \qquad \text{with}\quad h(r):=r^s_+ \z(r)
\ee
and $\z$ given in (\ref{eq:def-z}), and where we use the notation
\be\label{eq:def-I-u-v-new-section}
  I(u,v)(x):= \int_{\R^N}{(u(x)-u(y))(v(x)-v(y))}K_0(x,y)\, dy
\ee
for $u\in C^s_c(\R^N)$, $v\in C^{0,1}(\R^N)$ and $x \in \R^N$.
\end{proposition}

The remainder of this section is devoted to the proof of this proposition. For $k \in \N$, we define 
\begin{equation}
  \label{eq:def-ge-k}
g_k:= \n U_k \cdot X \Bigl( u   \Ds \z_k  -I(u,\z_k)\Bigr) \quad : \quad \Omega \to \R.
\end{equation}
For $\eps>0$, we put 
$$
\Omega^\eps= \{x \in \R^N\::\: |\delta(x)|< \eps\} \qquad \text{and}\qquad 
\Omega^\eps_+= \{x \in \R^N\::\: 0< \delta(x)< \eps\}= \{x \in \Omega\::\: \delta(x)< \eps\}.
$$
For every $\eps>0$, we then have 
\begin{equation}
  \label{eq:reduction-to-eps-collar-prelim-1}
\lim_{k\to \infty}\int_{\Omega \setminus \Omega^\eps} g_k\,dx =0.  
%=\int_{\Omega \setminus \Omega^\eps} \n U_k \cdot X \Bigl( u \Ds \z_k + I(u,\z_k)\Bigr)\,dx 
\end{equation}
To see this, we first note that $\z_k  \to 1$ pointwise on $\R^N \setminus \partial \Omega$, and therefore a.e. on $\R^N$. Moreover, choosing a compact neighborhood $K \subset \Omega$ of $\Omega \setminus \Omega^\eps$, we have 
$$
\Ds \z_k(x)= c_{N,s}\int_{\R^N \setminus K}\frac{1- \z_k(y)}{|x-y|^{N+2s}}dy \qquad \text{for $x \in \Omega \setminus \Omega^\eps$ and $k$ sufficiently large,}
$$
where $\frac{|1- \z_k(y)|}{|x-y|^{N+2s}}\le \frac{C}{1+|y|^{N+2s}}$ for $x \in \Omega \setminus \Omega^\eps,\: y\in \R^N \setminus K$ and $C>0$ independent of $x$ and $y$. Consequently, $\|\Ds \z_k\|_{L^\infty(\Omega \setminus \Omega^\eps)}$ remains bounded independently of $k$ and $\Ds \z_k \to 0$ pointwise on $\Omega \setminus \Omega^\eps$ by the dominated convergence theorem. Similarly, we see that $\|I(u,\z_k)\|_{L^\infty(\Omega \setminus \Omega^\eps)}$ remains bounded independently of $k$ and $I(u,\z_k) \to 0$ pointwise on $\Omega \setminus \Omega^\eps$. Consequently, we find that 
$$
\text{$\|g_k\|_{L^\infty(\Omega \setminus \Omega^\eps)}$ is bounded independently of $k$ and $g_k \to  0$ pointwise on $\Omega \setminus \Omega^\eps$.}
$$
Hence (\ref{eq:reduction-to-eps-collar-prelim-1}) follows again by the dominated convergence theorem. As a consequence, 
\begin{equation}
  \label{eq:reduction-to-eps-collar}
\lim_{k\to \infty}\int_{\Omega} g_k(x) \,dx = \lim_{k\to \infty}\int_{\Omega^\eps_+} g_k(x)\,dx
\qquad \text{for every $\eps>0$.}  
\end{equation}
Let, as before, $\nu: \partial \Omega \to \R^N$ denotes the unit interior normal vector field on $\Omega$. Since we assume that $\partial \Omega$ is of class $C^{1,1}$, the map $\nu$ is Lipschitz, which means that the derivative $d \nu: T \partial \Omega \to \R^N$ is a.e. well defined and bounded. Moreover, we may fix $\eps>0$ from now on with the property that the map 
\be\label{eq:Om-diffeo}
\Psi: \partial \O \times (-\e, \e) \to \Omega^{\e},\qquad (\sigma,r) \mapsto {\Psi}(\s,r)= \s + r \nu(\s)
\ee
is a bi-Lipschitz map with $\Psi(\partial \O \times (0, \e))= \Omega^{\e}_+$. In particular, $\Psi$ is a.e. differentiable, and
the variable $r$ is precisely the signed distance of the point ${\Psi}(\s,r)$ to the boundary $\partial \Omega$, i.e., 
\begin{equation}
  \label{eq:new-delta-equation}
\delta(\Psi(\s,r)) = r \qquad \text{for $\s \in \partial \O,\: 0 \le r < \eps.$}
\end{equation}
Moreover, for $0<\eps' \le \eps$, it follows from (\ref{eq:reduction-to-eps-collar}) that 
\begin{align}
\lim_{k\to \infty}\int_{\Omega} g_k\,dx&= \lim_{k\to \infty}\int_{\Omega^{\eps'}_+} g_k\,dx = \lim_{k\to \infty}\int_{\partial \O} \int_{0}^{\eps'}{\rm Jac}_{{\Psi}}(\s,r) g_k({\Psi}(\s,r))\,dr d\sigma \nonumber\\
&= \lim_{k\to \infty}\frac{1}{k} \int_{\partial \O} \int_{0}^{k \e'}j_k(\s,r) G_k(\s,r) \,dr d\sigma, \label{claim-first-reduction}
\end{align}
where we define  
\begin{equation}
  \label{eq:def-j-k-G-k}
j_k(\s,r) = {\rm Jac}_{{\Psi}}(\s,\frac{r}{k}) \quad \text{and}\quad G_k(\s,r)=g_k({\Psi}(\s,\frac{r}{k})) \qquad \text{for a.e. $\s \in \partial \O,\: 0 \le r < k \eps.$}
\end{equation}
We note that 
\begin{equation}
  \label{eq:j-k-est}
  \begin{aligned}
&\|j_k\|_{L^\infty(\partial \Omega \times [0,k\eps))} \le \|{\rm Jac}_{{\Psi}}\|_{L^\infty(\Omega_\eps)}< \infty \quad \text{for all $k$, and}\\
&\lim_{k \to \infty} j_k(\sigma,r)= {\rm Jac}_{{\Psi}}(\s,0)=1 \quad \text{for a.e. $\s \in \partial \Omega$, $r>0$.}
   \end{aligned} 
\end{equation}
By definition of the functions $g_k$ in (\ref{eq:def-ge-k}), we may write
\begin{equation}
\label{G-k-splitting}
G_k(\s,r) = G_{k}^0(\s,r)[G_{k}^1(\s,r) - G_{k}^2(\s,r)] \qquad \text{for $\s \in \partial \O,\: 0 \le r < k \eps$}
\end{equation}
with 
\begin{equation}
\begin{aligned}
G_k^0(\s,r) &= [\nabla U_k \cdot X]({\Psi}(\s,\frac{r}{k}))\\  
              G_k^1(\s,r) &= [u   \Ds \z_k]({\Psi}(\s,\frac{r}{k}))  && \qquad \text{and}\\
                            G_k^2(\s,r) &= I(u,\z_k)({\Psi}(\s,\frac{r}{k})).
\end{aligned} \label{h-0-1-2} 
\end{equation}
In order to analyze the limit in (\ref{claim-first-reduction}) for suitable $\e' \in (0,\e]$, we provide estimates for the functions $G_k^0, G_k^1, G_k^2$ separately in the following. We start with an estimate for $G_k^0$ given by the following lemma.
\begin{lemma}
\label{first-function-est}
Let $\alpha \in (0,1)$ be given by Lemma~\ref{reg-prop-minimizers}. Then we have 
\begin{equation}
  \label{eq:h-k-0-bound}
k^{s-1}|G_k^0(\s,r)| \le C(r^{s-1}+r^{s-1+\alpha}) \qquad \text{for $k \in \N$, $0 \le r < k\eps$}
\end{equation}
with a constant $C>0$, and 
\begin{equation}
  \label{eq:h-k-0-limit}
\lim_{k \to \infty} k^{s-1} G_k^0(\s,r) = h'(r) \psi(\s)
[X(\s) \cdot \nu(\s)] \quad \text{for $\s \in \partial \Omega$, $r>0$}
\end{equation}
with the function $r \mapsto h(r)= r^s_+ \zeta(r)$ given in \eqref{eq:-kappa-s-implicit-new-section}.   
\end{lemma}

\begin{proof}
Since $u= \psi \delta^s$, we have 
$$
\nabla u = s \delta^{s-1}  \psi \nabla \delta + \delta^s \nabla \psi = s \delta^{s-1}  \psi \nabla \delta 
+ O(\delta^{s-1+\alpha})\qquad \text{in $\Omega$}
$$
by Lemma~\ref{reg-prop-minimizers}, and therefore, since $\z_k = \z \circ (k \delta)$ by (\ref{eq:def-z-k-etc}), 
$$
\nabla U_k = \nabla \Bigl(u \z_k \Bigr)= \Bigl(s \z \circ (k \delta)  + k \delta \z' \circ (k \delta)\Bigr) \psi \delta^{s-1}\nabla \delta  + O(\delta^{s-1+\alpha}) \qquad \text{in $\Omega$.}
$$    
Consequently, by (\ref{eq:new-delta-equation}) we have
$$
\bigl[\bigl(\nabla U_k\bigr) \circ {\Psi}\bigr](\s,\frac{r}{k}) =  
\Bigl(s \z(r)  +  r \z'(r)\Bigr) \psi(\s + \frac{r}{k}\nu(\s)) \bigl(\frac{r}{k}\bigr)^{s-1} 
\nabla \delta(\s + \frac{r}{k}\nu(\s))  + O\Bigl(\bigl(\frac{r}{k}\bigr)^{s-1+\alpha}\Bigr) 
$$
for $\s \in \partial \Omega$, $0 \le r < \eps$ with $O(r^{s-1+\alpha})$ independent of $k$, and therefore
\begin{align*}
G_k^0(\s,r)&= \Bigl(s \z(r)  + r \z'(r)\Bigr) \psi(\s + \frac{r}{k}\nu(\s))
\nabla \delta(\s + \frac{r}{k}\nu(\s)) \cdot X(\s + \frac{r}{k}\nu(\s))k^{1-s} r^{s-1}\\
&+ k^{1-s-\alpha} O(r^{s-1+\alpha}) \qquad \text{for $\s \in \partial \Omega$, $0 \le r < k \eps$.}
\end{align*}
Since $\alpha>0$, we deduce that 
\begin{align*}
k^{s-1}G_k^0(\s,r)  &\to \Bigl(s \z(r)  + r \z'(r)\Bigr) \psi(\s)
\nabla \delta(\s) \cdot X(\s)r^{s-1}=h'(r) \psi(\s)
X(\s) \cdot \nu(\s) \qquad \text{as $k \to \infty$}
\end{align*}
for $\s \in \partial \Omega$, $r>0$, while 
$$
k^{s-1}|G_k^0(\s,r)| \le C(r^{s-1}+r^{s-1+\alpha}) \qquad \text{for $k \in \N$, $0 \le r < k\eps$}
$$
with a constant $C>0$ independent of $k$ and $r$, as claimed.
\end{proof}

Next we consider the functions $G_k^1$ defined in (\ref{h-0-1-2}), and we first state the following estimate.

 \begin{proposition}
\label{delta-s-z-k-conv}
There exists $\eps'>0$ with the property that 
\begin{equation}
  \label{eq:delta-s-z-k-concl-bound}
|k^{-2s} \Ds \z_k({\Psi}(\sigma, \frac{r}{k}))|\le \frac{C}{1+r^{1+2s}}\qquad \text{for $k \in \N$, $0 \le r < k \eps'$}
\end{equation}
with a constant $C>0$. Moreover, 
\begin{equation}
  \label{eq:delta-s-z-k-concl-limit}
\lim_{k \to \infty}k^{-2s} \Ds \z_k({\Psi}(\sigma, \frac{r}{k})) =  (-\Delta)^s \z (r) \quad \text{for $\s \in \partial \Omega$, $r>0$.}
\end{equation}
\end{proposition}

Before giving the somewhat lengthy proof of this proposition, we infer the following corollary related to the functions $G_k^1$. 

\begin{corollary}
\label{corol-second-function-est}
There exists $\eps'>0$ with the property that    
\begin{equation}
  \label{eq:h-2-bound}
|k^{-s}G_k^1(\sigma,r)|\le \frac{C r^s}{1+r^{1+2s}}\qquad \text{for $k \in \N$, $0 \le r < k \eps'$}
\end{equation}
with a constant $C>0$. Moreover, 
\begin{equation}
  \label{eq:h-2-limit}
\lim_{k \to \infty}k^{-s}G_k^1(\sigma,r) = \psi(\sigma) r^s (-\Delta)^s \z (r) \quad \text{for $\s \in \partial \Omega$, $r>0$.}
\end{equation}
\end{corollary}

\begin{proof}
Since $u = \psi \delta^s$ we have $u({\Psi}(\sigma,\frac{r}{k}))= k^{-s} \psi(\sigma + \frac{r}{k} \nu(\sigma)) r^s$ for $k \in \N$, $0 \le r < k\eps$, and 
$$
\lim_{k \to \infty}k^{s}u({\Psi}(\sigma,\frac{r}{k})) = \psi(\sigma) r^s \quad \text{for $\s \in \partial \Omega$, $r>0$.}
$$
Since moreover $\|\psi\|_{L^\infty(\Omega_\eps)}<\infty$, the claim now follows from Proposition~\ref{delta-s-z-k-conv} by recalling the definition in $G_k^1$ in (\ref{h-0-1-2}).
\end{proof}

We now turn to the proof of Proposition~\ref{delta-s-z-k-conv}, and we need some preliminary considerations. Since $\partial \Omega$ is of class $C^{1,1}$ by assumption, there exists an open ball $B \subset \R^{N-1}$ centered at the origin and, for every $\s\in \de\O$, a parametrization $f_\s: B \to \de\O$
of class $C^{1,1}$ with the property that $f_\s(0)=\s$ and $d f_\s(0): \R^{N-1} \to \R^N$ is a linear isometry. For $z \in B$ we then have 
$$
f_\s(z)-f_\s(0)= df_\s(0)z + O(|z|^2)
$$
and therefore 
\begin{align}
  \label{eq:f-s-first-exp}
&|f_\s(0)-f_\s(z)|^2 = |df_\s(0)z|^2 + O(|z|^3) = |z|^2 + O(|z|^3),\\
  \label{eq:f-s-second-exp}
&(f_\s(0)-f_\s(z))\cdot \nu(\s) = - df_\s(0)z \cdot \nu(\s)+ O(|z|^2)= O(|z|^2),
\end{align}
where we used in (\ref{eq:f-s-second-exp}) that $d f_\s(0)z$ belongs to the tangent space $T_\s \partial \Omega = \{\nu(\s)\}^\perp$. Here and in the following, the term $\calO(\tau)$ stands for a function depending on $\tau$ and possibly other quantities but satisfying $|\calO(\tau)| \le C \tau$ with a constant $C>0$.

Recalling the definition of the map $\Psi$ in \eqref{eq:Om-diffeo} and writing $\nu_\sigma(z):= \nu(f_\s(z))$ for $z \in B$, we now define
\be
\Psi_\s: (-\e,\e) \times B \to \Omega^\eps,\qquad \Psi_\s(r,z) = \Psi(f_\s(z),r) = 
f_\s(z) + r \nu_\sigma(z).
\ee
Then $\Psi_\s$ is a bi-Lipschitz map which maps $(-\e,\e) \times B$ onto a neighborhood of $\sigma$. Consequently, there exists $\eps' \in (0,\frac{\eps}{2})$ with the property that 
\begin{equation}
  \label{eq:distance-sigma-ineq}
|\sigma-y| \ge 3\eps' \qquad \text{for all $y \in \R^N \setminus \Psi_\s((-\e,\e) \times B)$.}
\end{equation}
Moreover, $\eps'$ can be chosen independently of $\sigma \in \partial \Omega$. 

Coming back to the proof of Proposition~\ref{delta-s-z-k-conv}, we now write, for $\sigma \in \partial \Omega$ and $r \in [0,k\eps')$,  
\begin{equation}
  \label{eq:delta-s-splitting}
\Ds \z_k(\Psi(\s,\frac{r}{k})) = c_{N,s} \Bigl( A_k(\sigma,r)+ B_k(\sigma,r) \Bigr)
\end{equation}
with 
$$
A_{k}(\sigma,r) := \int_{ \Psi_\s((-\e,\e) \times B)} \frac{\z(r)-\z_k(y)}{|\Psi(\s,\frac{r}{k}) - y|^{N+2s}}\, dy
$$
and 
$$
B_k(\sigma,r) := \int_{\R^N \setminus \Psi_\s((-\e,\e) \times B)} \frac{\z(r)-\z_k(y)}{|\Psi(\s,\frac{r}{k})- y|^{N+2s}}\, dy.
$$
Here we used that $\z_k(\Psi(\s,\frac{r}{k}))= \z(r)$ for $\sigma \in \partial \Omega$, $r \in [0,k \eps')$ by (\ref{eq:new-delta-equation}) and the definition of $\z_k$.  We first provide a rather straightforward estimate for the functions $B_k$. 
\begin{lemma}
  \label{B-k-new-lemma}
We have   
\begin{equation}
  \label{eq:B-estimate-majorization}
k^{-2s}|B_k(\sigma,r)| \le \frac{C}{1+r^{1+2s}}  
\qquad \quad \text{for $k \in \N$, $0 \le r < k \eps'$, $\sigma \in \partial \Omega$}
\end{equation}
with a constant $C>0$ and 
\begin{equation}
  \label{eq:B-estimate-limit}
\lim_{k \to \infty}k^{-2s}|B_k(\sigma,r)|=0 \qquad \text{for every $\sigma \in \Omega$, $r\ge 0$.}
\end{equation}
\end{lemma}

\begin{proof}
By (\ref{eq:distance-sigma-ineq}) and since $r < k\eps'$, we have 
$$
|\Psi(\s,\frac{r}{k})- y|= |\sigma - y + \frac{r}{k}\nu(\sigma)|\ge 
|\sigma - y|- \frac{r}{k} \ge \frac{|\sigma - y|}{3} +\eps' \quad \text{for $y \in \R^N \setminus \Psi_\s((-\e,\e) \times B)$.}
$$
Recalling that $\z = 1 -\rho$, $\:\z_k = 1-\rho_k$ and that $\rho_k$ is supported in $\Omega^{\text{\tiny{$\frac{2}{k}$}}}$, we thus estimate
\begin{align*}
|B_k(\sigma,r)| &\le \int_{\R^N \setminus \Psi_\s((-\e,\e) \times B)} \frac{|\rho(r)-\rho_k(y)|}{|\Psi(\s,\frac{r}{k})- y|^{N+2s}}\, dy \\
&\le 3^{N+2s} |\rho(r)| \int_{\R^N} \Bigl(|\sigma - y| +3 \eps'\Bigr)^{-N-2s}dy 
+ \bigl(\eps'\bigr)^{-N-2s} \int_{\R^N} |\rho_k(y)| \, dy\\
&\le C \Bigl(|\rho(r)| +|\Omega^{\text{\tiny{$\frac{2}{k}$}}}|\Bigr) \le C \Bigl(|\rho(r)| + k^{-1}\Bigr).
\end{align*}
Here and in the following, the letter $C$ stands for various positive constants. This estimate readily yields (\ref{eq:B-estimate-limit}). Moreover, 
$$
k^{-2s}|B_k(\sigma,r)| \le C k^{-2s}\Bigl(|\rho(r)| + k^{-1}\Bigr) \le \frac{C}{1+r^{1+2s}} + k^{-1-2s} \le \frac{C}{1+r^{1+2s}}  
$$
for $k \in \N$, $0 \le r < k \eps'$, $\sigma \in \partial \Omega$, as claimed in (\ref{eq:B-estimate-majorization}).
\end{proof}

To complete the proof of Proposition~\ref{delta-s-z-k-conv}, it thus remains to consider the functions $A_{k}$ in the following. For this, we need the following additional estimates for the maps $\Psi_\s$, $\sigma \in \partial \Omega$. We note here that $\Psi_\s$ is a.e. differentiable since it is Lipschitz, so the Jacobian determinant ${\rm Jac}_{{\Psi_\s}}$ is a.e. well-defined on $(-\e,\e) \times B$.

\begin{lemma}
  \label{new-lemma-Psi-sigma}
  There exists a constant $C_0$ with the property that for every $\sigma \in \partial \Omega$ we have the following estimates:\\[0.3cm]
  (i) $\quad |{\rm Jac}_{{\Psi_\s}}(r,z)| \le C_0$ for a.e. $r \in (-\eps,\eps)$, $z \in B$;\\[0.3cm]   
  (ii) $\quad |{\rm Jac}_{{\Psi_\s}}(r,z)-1| \le C_0(|r|+|z|)$ for a.e. $r \in (-\eps,\eps)$, $z \in B$;\\[0.3cm] 
  (iii) $\quad |{\rm Jac}_{{\Psi_\s}}(r+t,z)-{\rm Jac}_{{\Psi_\s}}(r-t,z)| \le C_0|t|$ for a.e. $r \in (-\eps,\eps)$, $z \in B$, $t \in (-\eps-r,\eps-r)$;\\[0.3cm]
  Moreover, for $\sigma \in \partial \Omega$, $r \in (-\eps,\eps)$, $z \in B$, $t \in (-\eps-r,\eps-r)$ we have\\[0.3cm]
  (iv) $\quad \frac{1}{C_0} \bigl(t^2 + |z|^2\bigr)^{\frac{1}{2}} \le  |{\Psi_\s}(r,0)- {\Psi_\s}(r+t,z)|  \le C_0 \bigl(t^2 + |z|^2\bigr)^{\frac{1}{2}}$,\\[0.3cm]
and for $\sigma \in \partial \Omega$, $r \in (-\eps,\eps)$, $t \in (-\eps-r,\eps-r) \setminus \{0\}$ and $z \in \frac{1}{|t|}B$  we have\\[0.3cm]
  (v) $\quad \Bigl|\frac{|{\Psi_\s}(r,0)- {\Psi_\s}(r+t ,|t| z)|^{2}}{t^2}-(1+|z|^2)\Bigr| \le C_0 (|t|+ |r| + |tz|)|z|^2$;\\[0.3cm]
  (vi) $\quad \Bigl|\bigl|{\Psi_\s}(r,0)- {\Psi_\s}(r+t ,|t| z)\bigr|^{-N-2s}- \bigl|{\Psi_\s}(r,0)- {\Psi_\s}(r-t ,|t| z)\bigr|^{-N-2s}\Bigr|$\vspace{0.3cm}
  
  \hspace{10cm} $\le C_0 |t|^{1-N-2s}(1+|z|^2)^{-\frac{N+2s}{2}}$.
\end{lemma}

\begin{proof}
 The inequalities (i) and (iv) are direct consequences of the fact that $\Psi_\s$ is bi-Lipschitz. In particular, if $C_0$ is a Lipschitz constant for $\Psi_\s^{-1}$, we have 
  $$
  \bigl(t^2 + |z|^2\bigr)^{\frac{1}{2}} = |(-t,z)|= |(r,0) -(r+t,z)| \le C_0 |{\Psi_\s}(r,0)- {\Psi_\s}(r+t,z)|
$$
  for $\sigma \in \partial \Omega$, $r \in (-\eps,\eps)$, $z \in B$ and $t \in (-\eps-r,\eps-r)$, so the first inequality in (iv) follows. By making $C_0$ larger if necessary so that it is also a Lipschitz constant for $\Psi_\s$, we then deduce the second inequality in (iv).
  
  To see (ii) and (iii), we note that $d \Psi_\s$ is a.e. given by
$$
d \Psi_\s (r,z)(r',z') = [d f_\sigma(z)+ r  d \nu_\sigma(z)]z' + r' \nu_\sigma(z) 
$$
for $(r,z) \in (-\eps,\eps) \times B$, $(r',z') \in \R \times \R^{N-1}$, which implies that 
$$
[d \Psi_\s (r,z)-d \Psi_\s (0,0)] (r',z') = [d f_\sigma(z)-df_\sigma(0)]z'+ r d \nu_\sigma(z) +r' \bigl(\nu_\sigma(z)-\nu_\sigma(0)\bigr)
$$
and
$$
[d \Psi_\s (r+t,z)-d \Psi_\s (r-t,z)] (r',z') = 2t d \nu_\sigma(z) z'.
$$
Since $d f_\sigma$, $\nu_\sigma$ are Lipschitz functions on $B$, $d\nu_\s$ is a bounded function on $B$ and the determinant is a locally Lipschitz continuous function on the space of linear endomorphisms of $\R^N$, it follows that
$$
|{\rm Jac}_{{\Psi_\s}}(r,z)-{\rm Jac}_{{\Psi_\s}}(0,0)| \le C_0(|r|+|z|) \quad \text{and}\quad 
|{\rm Jac}_{{\Psi_\s}}(r+t,z)-{\rm Jac}_{{\Psi_\s}}(r-t,z)| \le C_0|t|
$$
for a.e. $r \in (-\eps,\eps)$, $z \in B$, $t \in (-\eps-r,\eps-r)$. Moreover, ${\rm Jac}_{{\Psi_\s}}(0,0) = 1$ since the map
$$
\R \times \R^{N-1} \to \R^N, \qquad (r',z') \mapsto d{\Psi_\s}(0,0)(r',z') = d f_\sigma(0)z' + r' \nu_\sigma(0)
$$
is an isometry. Hence (ii) and (iii) follow.

To see (v) and (vi), we note that by definition of $\Psi_\s$ we have 
$$
{\Psi_\s}(r,0)- {\Psi_\s}(r+t,z)=f_\s(0)-f_\s(z)-t \nu_\s(0)+(r+t)(\nu_\s(0)-\nu_\s(z)) 
$$
for $z \in B$, $r \in (0,\e')$ and $t \in (-\eps-r,\e-r)$. Using moreover that $(\nu_\s(0)-\nu_\s(z))\cdot \nu_\s(0)=\frac{1}{2}|\nu_\s(0)-\nu_\s(z)|^2$, we get 
\begin{align}
&|{\Psi_\s}(r,0)- {\Psi_\s}(r+t,z)|^2 
=t^2+ |f_\s(0)-f_\s(z)|^2 +(r+t)^2|\nu_\s(0)-\nu_\s(z)|^2 \nonumber\\
&- 2t(f_\s(0)-f_\s(z))\cdot \nu_\s(0)  -t(r+t)|\nu_\s(0)-\nu_\s(z)|^2 +2(r+t)(f_\s(0)-f_\s(z))\cdot (\nu_\s(0)-\nu_\s(z)) \nonumber\\
&=t^2+ |f_\s(0)-f_\s(z)|^2 +r(r+t)|\nu_\s(0)-\nu_\s(z)|^2 \nonumber\\
&\qquad \qquad \qquad - 2t(f_\s(0)-f_\s(z))\cdot \nu_\s(0)  +2(r+t)(f_\s(0)-f_\s(z))\cdot (\nu_\s(0)-\nu_\s(z)) \nonumber \\
&=t^2+ |z|^2 + \bigl[|z| m_\s(z) +r(r+t) n_\sigma(z)-2t p_\s(z) +2(r+t)q_\sigma(z) \bigr]|z|^2
 \label{prelim-expansion-psi-s-diff}
\end{align}
for $z \in B$, $r \in (-\e,\e)$ and $t \in (-\eps-r,\e-r)$ with the functions 
$$
m_\s(z) = \frac{|f_\s(0)-f_\s(z)|^2-|z|^2}{|z|^3}, \quad n_\s(z)= \frac{|\nu_\s(0)-\nu_\s(z)|^2}{|z|^2},\quad
p_\s(z)= \frac{(f_\s(0)-f_\s(z))\cdot \nu_\s(0)}{|z|^2}
$$
and
$$
q_\s(z) = \frac{(f_\s(0)-f_\s(z))\cdot (\nu_\s(0)-\nu_\s(z))}{|z|^2}, \qquad z \in B \setminus \{0\},
$$
which are all bounded as a consequence of the Lipschitz continuity of $f_\s$ and $\nu_\s$ and of (\ref{eq:f-s-first-exp}) and (\ref{eq:f-s-second-exp}). We deduce that
\begin{align*}
  &\Bigl|\frac{|{\Psi_\s}(r,0)- {\Psi_\s}(r+t ,|t| z)|^{2}}{t^2}-(1+|z|^2)\Bigr|\\
  &= \Bigl||tz| m_\s(|t|z) +r(r+t) n_\sigma(|t|z)-2t p_\s(|t|z) +2(r+t)q_\sigma(|t|z) \Bigr| |z|^2 \le C_0(|tz|+ |r| + |t|)|z|^2   
\end{align*}
for $\sigma \in \partial \Omega$, $r \in (-\eps,\eps)$, $t \in (-\eps-r,\eps-r) \setminus \{0\}$ and $z \in \frac{1}{|t|}B$ if $C_0$ is chosen sufficiently large, as claimed in (v).

For the proof of (vi), we now set $w_\sigma(r,t,z):= \frac{1}{t^2}|{\Psi_\s}(r,0)- {\Psi_\s}(r+t,|t|z)|^2$, and we note that
$$
w_\sigma(r,t,z) \ge \frac{1+|z|^2}{C_0^2}\quad \text{for $\sigma \in \partial \Omega$, $r \in (-\eps,\eps)$, $t \in (-\eps-r,\eps-r) \setminus \{0\}$, $z \in \frac{1}{|t|}B$}
$$
by (iv). Moreover, from (\ref{prelim-expansion-psi-s-diff}) we infer that
\begin{equation*}
  \Bigl|w_\sigma(r,t,z) -w_\sigma(r,-t,z) \Bigr|= \Bigl| 2 r t  n_\sigma(|t|z) +4t\bigl( q_\sigma(|t|z)-p_\s(|t|z)\bigr) \Bigr| |z|^2 \le C_0 |t||z|^2 
\end{equation*}
for $\sigma \in \partial \Omega$, $r \in (-\eps,\eps)$, $t \in (-\eps-r,\eps-r) \setminus \{0\}$ and $z \in \frac{1}{|t|}B$ if $C_0$ is made larger if necessary.  Using these estimates together with the mean value theorem, we get that, for some $\tau= \tau(\sigma,r,t,z)$ with $-t < \tau< t$, 
\begin{align*}
  \Bigl|&|{\Psi_\s}(r,0)- {\Psi_\s}(r+t ,|t|z)|^{-N-2s}-|{\Psi_\s}(r,0)- {\Psi_\s}(r-t ,|t|z)|^{-N-2s}\Bigr|\\
        &= |t|^{-N-2s}\Bigl|w_\sigma(r,t,z)^{-\frac{N+2s}{2}} -w_\sigma(r,-t,z)^{-\frac{N+2s}{2}} \Bigr|\\
&=  \frac{(N+2s)|t|^{-N-2s} }{2} w_\sigma(r,\tau,z)^{-\frac{N+2s+2}{2}} \Bigl|w_\sigma(r,t,z) -w_\sigma(r,-t,z) \Bigr|\\
        &\le C_0|t|^{1-N-2s} (1+|z|^2)^{-\frac{N+2s+2}{2}}|z|^2 \le C_0|t|^{1-N-2s} (1+|z|^2)^{-\frac{N+2s}{2}} \label{diff-first-est} 
\end{align*} 
for $z\in B$, $r\in (0,\e')$ and $t \in (-\e+r,\e-r)$ after making $C_0$ larger if necessary, as claimed in (vi).
\end{proof}

We now have all the tools to study the quantity $A_{k}(\sigma,r)$ in \eqref{eq:delta-s-splitting}.

\begin{lemma}
  \label{A-k-new-lemma}
We have   
\begin{equation}
  \label{eq:A-estimate-majorization}
k^{-2s}|A_k(\sigma,r)| \le \frac{C}{1+r^{1+2s}}  
\qquad \quad \text{for $k \in \N$, $0 \le r < k \eps'$, $\sigma \in \partial \Omega$}
\end{equation}
with a constant $C>0$ and 
\begin{equation}
  \label{eq:A-estimate-limit}
\lim_{k \to \infty}k^{-2s}A_k(\sigma,r)=\frac{(-\Delta)^s\z(r)}{c_{N,s}} \qquad \text{for every $\sigma \in \Omega$, $r\ge 0$.}
\end{equation}
\end{lemma}

\begin{proof}
For $\sigma \in \partial \Omega$ and $0 < r < k \e'$, we write, with a change of variables,
\begin{align}
  &A_{k}(\sigma,r) \label{change-of-variables-A}\\
  &= \int_{ \Psi_\s((-\e,\e) \times B)} \frac{\z(r)-\z_k(y)}{|\Psi(\s,\frac{r}{k}) - y|^{N+2s}}\, dy= \int_{-\e}^{\e} \int_{B}{\rm Jac}_{\Psi_\s}(\tilde r,z)\: \frac{\z(r)-\z(k \tilde r)}{|\Psi_\s(\frac{r}{k},0) - \Psi_\s(\tilde r,z)|^{N+2s}}\,d z d \tilde r \nonumber\\
  &= \frac{1}{k}\int_{-k\e-r}^{k\e-r} \int_{B}{\rm Jac}_{\Psi_\s}(\frac{r+t}{k},z)\: \frac{\z(r)-\z(r+t)}{|\Psi_\s(\frac{r}{k},0) - \Psi_\s(\frac{r+t}{k},z)|^{N+2s}}\,d z d t \nonumber\\
  &= \int_{-k\e-r}^{k\e-r} \frac{|t|^{N-1}}{k^N} \int_{\frac{k}{|t|}B}{\rm Jac}_{\Psi_\s}(\frac{r+t}{k},\frac{|t|z}{k})\: \frac{\z(r)-\z(r+t)}{|\Psi_\s(\frac{r}{k},0) - \Psi_\s(\frac{r+t}{k},\frac{|t|z}{k})|^{N+2s}}\,d z d t\nonumber\\
  &= k^{2s}\int_{\R}\frac{\z(r)-\z(r+t)}{|t|^{1+2s}}\cK_k(r,t) dt \nonumber
\end{align}
with the kernels $\cK_k: (0,k\e') \times \R  \to \R$ defined by  
$$
\cK_k(r,t)= \left\{
  \begin{aligned}
   & \Bigl(\frac{|t|}{k}\Bigr)^{N+2s} \int_{\frac{k}{|t|}B} \frac{ {\rm Jac}_{{\Psi_\s}}(\frac{r+t}{k},\frac{|t|z}{k})}{\bigl|{\Psi_\s}(\frac{r}{k},0)-{\Psi_\s}(\frac{r+t}{k},\frac{|t|z}{k}) \bigr|^{N+2s}}dz,&&\quad t\in (-k\eps-r,k\eps-r), \\
   &0,&&\quad t \not \in (-k\eps-r,k\eps-r).
    \end{aligned}
\right.
$$
Consequently, 
\begin{equation}
  \label{eq:A-k-splitting}
A_{k}(\sigma,r)= k^{2s}\Bigl(
J^1_k(\s,r)+ J^2_k(\s,r)\Bigr)
\end{equation}
with 
$$
J^1_k(\s,r):=\frac{1}{4 } \int_{\R} \frac{2\z(r)-\z(r+t)-\z(r-t)}{|t|^{1+2s}} \bigl(\cK_k(r,t) +\cK_k(r,-t)\bigr) dt
$$
and 
$$
J^2_k(\s,r):=-\frac{1}{4} \int_{\R} \frac{\z(r+t)-\z(r-t)}{|t|^{2s}}\,\frac{\cK_k(r,t)- \cK_k(r,-t)}{|t|} dt.
$$
By Lemma~\ref{new-lemma-Psi-sigma}(i),(iv) and the definition of $\cK_k$, we have 
\be \label{eq:est-cK1}
|\cK_k(r,t)|   \le C_0^{N+2s+1} \int_{\frac{k}{|t|}B} \Bigl(1+ \Bigl|z|^2\Bigr)^{-\frac{N+2s}{2}}dz \le  C_0^{N+2s+1} a_{N,s}
\ee
for $r \in (-k\e',k\e')$ and $t \in \R \setminus \{0\}$ with
\be
\label{def-a-n-s}
a_{N,s}:=\int_{\R^{N-1}}(1+|z|^2)^{-\frac{N+2s}{2}}dz < \infty.
\ee
Moreover, by Lemma~\ref{new-lemma-Psi-sigma}(i)(ii),(iv),(v) and the dominated convergence theorem, we have 
\be \label{eq:est-cK2}
\lim_{k\to \infty}\cK_k(r,t)= \int_{\R^{N-1}}(1+|z|^2)^{-\frac{N+2s}{2}}\,dz = a_{N,s}
\qquad \text{for every $r \ge 0$, $t \in \R \setminus \{0\}$.}
\ee
Using \eqref{eq:est-cK1} and the fact that $\rho=1-\z \in C^\infty_c(\R)$, we obtain the estimate
\begin{align}
| J^1_k(\s,r)| &\leq C \int_{\R}   {\frac{|2\z( r)-\z(r+t)-\z(r-t)|}{|t|^{1+2s}}} \,dt  \nonumber\\
&= C \int_{\R}   {\frac{|2\rho(r)-\rho(r+t)-\rho(r-t)|}{|t|^{1+2s}}} \,dt \leq  \frac{C}{1+r^{1+2s}}  \label{eq:J1}
\end{align}
for $k \in \N$, $r\in (0,k\e')$ and $\sigma \in \partial \Omega$. Here and in the following, the letter $C>0$ stands for different positive constants.  Moreover, by \eqref{eq:est-cK1}, \eqref{eq:est-cK2} and the dominated convergence theorem, we find that 
\begin{equation}
\lim_{k \to \infty} J^1_k(\s,r) 
=  \frac{a_{N,s}}{2} \int_{\R}\frac{2\z(r)-\z(r+t)-\z(r-t)}{|t|^{1+2s}}dt = \frac{a_{N,s}}{c_{1,s}} (-\Delta)^s\z(r)= \frac{(-\Delta)^s\z(r)}{c_{N,s}}. \label{eq:J2}
\end{equation}
Here we have used the fact that 
\be\label{eq:bNs-aNs}
c_{N,s}a_{N,s}=c_{1,s},
\ee
see e.g. \cite{FW-half}.

Next we deal with $J^2_k(\s,r)$, and for this we have to estimate the kernel differences $|\cK_k(r,t)- \cK_k(r,-t)|$. 
By Lemma~\ref{new-lemma-Psi-sigma}(i),(iii),(iv) and (vi), we have  
$$
\Bigl| \frac{ {\rm Jac}_{{\Psi_\s}}(\frac{r+t}{k},\frac{|t|z}{k})}{\bigl|{\Psi_\s}(\frac{r}{k},0)-{\Psi_\s}(\frac{r+t}{k},\frac{|t|z}{k}) \bigr|^{N+2s}}-\frac{ {\rm Jac}_{{\Psi_\s}}(\frac{r-t}{k},\frac{|t|z}{k})}{\bigl|{\Psi_\s}(\frac{r}{k},0)-{\Psi_\s}(\frac{r-t}{k},\frac{|t|z}{k}) \bigr|^{N+2s}}\Bigr|\le C \Bigl(\frac{|t|}{k}\Bigr)^{1-N-2s}(1+|z|^2)^{-\frac{N+2s}{2}} 
$$
for $z\in \frac{k}{|t|}B$, $r\in (0,k\e')$ and $t \in (-k\e+r,k\e-r)$ and therefore
\begin{equation}
  \frac{|\cK_k(r,t)- \cK_k(r,-t)|}{|t|} \le \frac{C}{k} \int_{\frac{k}{|t|}B}(1+|z|^2)^{-\frac{N+2s}{2}} \,dz \le \frac{C}{k}
  \int_{\R^{N-1}}(1+|z|^2)^{-\frac{N+2s}{2}} \,dz \le \frac{C}{k} \label{eq:est-cK1-0-new}\
\end{equation}
for $r\in (0,k\e') $ and $t \in (-k\e+r,k\e-r)$. Moreover, by definition we have 
\begin{equation}
\label{eq:est-cK1-0-0-new}
  |\cK_k(r,t)- \cK_k(r,-t)|= 0 \qquad \text{for $t \in \R \setminus (-k\e-r,k\e+r)$,}
\end{equation}
while for $t \in (-k\e-r,-k\e+r) \cup 
(k\e-r,k\e+r)$ we have $|t|\ge k\e -\eps' \ge \frac{k\eps}{2}$ and therefore, similarly as in \eqref{eq:est-cK1}, 
\be \label{eq:est-cK1-new-1}
\frac{|\cK_k(r,t)|}{|t|}    \le \frac{C}{|t|} \int_{\frac{k}{|t|}B} \bigl(1+ |z|^2\bigr)^{-\frac{N+2s}{2}} dz \le
\frac{C}{k} \int_{\frac{2}{\eps}B} \bigl(1+ |z|^2\bigr)^{-\frac{N+2s}{2}} dz \le \frac{C}{k}.
\ee
Note here that the constant $C>0$ on the RHS depends on $\eps$, but this is not a problem. Combining (\ref{eq:est-cK1-0-new}), (\ref{eq:est-cK1-0-0-new}), \eqref{eq:est-cK1-new-1} and using that $\rho=1-\z\in C^\infty_c(\R)$, we get 
\begin{align*}
&|J^2_k(\s,r)| \le \frac{1}{4 } \int_{\R} \frac{|\z(r+t)-\z(r-t)|}{|t|^{2s}} \frac{|\cK_k(r,t)- \cK_k(r,-t)|}{|t|} dt\\ 
                      &\le \frac{C}{k}  \int_{\R} \frac{|\z(r+t)-\z(r-t)|}{t^{2s}}d t = \frac{C}{k}  \int_{\R} \frac{|\rho(r+t)-\rho(r-t)|}{t^{2s}} d t \le \frac{C(1+r)^{-2s}}{k}
\end{align*}
for $k \in \N$, $\sigma \in \partial \Omega$ and $0 \le r <k \e'$. Hence
\begin{equation}
  \label{eq:J-2-concl-bound}
|J^2_k(\s,r)|  \le \frac{C}{1+r^{1+2s}}\qquad \text{for $k \in \N$, $0 \le r < k \eps'$}
\end{equation}
and 
\begin{equation}
  \label{eq:J-2-concl-limit}
\lim_{k \to \infty}|J^2_k(\s,r)| = 0 \qquad \text{for all $r \ge 0$.}
\end{equation}
Now (\ref{eq:A-estimate-majorization}) follows by combining (\ref{eq:A-k-splitting}), (\ref{eq:J1}) and (\ref{eq:J-2-concl-bound}).
Moreover, (\ref{eq:A-estimate-limit}) follows by combining (\ref{eq:A-k-splitting}), (\ref{eq:J2}) and (\ref{eq:J-2-concl-limit}). \end{proof}

\begin{proof}[Proof of Proposition~\ref{delta-s-z-k-conv}]
The proof is completed by combining (\ref{eq:delta-s-splitting}) with Lemmas~\ref{B-k-new-lemma} and~\ref{A-k-new-lemma}. 
\end{proof}

It finally remains to estimate the function $G_k^2$ in (\ref{G-k-splitting}). 

 \begin{lemma}
\label{I-s-z-k-conv}
There exists $\eps'>0$ with the property that the function $G_k^2$ defined in (\ref{h-0-1-2}) satisfies  
\begin{equation}
  \label{eq:I-s-z-k-concl-bound}
|k^{-s} G_k^2(\sigma,r)| \le \frac{C}{1+r^{1+s}}\qquad \text{for $k \in \N$, $0 \le r < k \eps'$, $\sigma \in \partial \Omega$}
\end{equation}
with a constant $C>0$. Moreover, 
\begin{equation}
  \label{eq:I-s-z-k-concl-limit}
\lim_{k \to \infty} k^{-s} G_k^2(\sigma,r) = \psi(\s)\tilde I(r)
\end{equation}
with 
$$
\tilde I(r) = c_{1,s} \int_{\R}\frac{ \bigl(r^s_+ - (r+t)^s_+ \bigr)\bigl(\z(r)-\z(r+t)\bigr)}{|t|^{1+2s}}dt.
$$
\end{lemma}

\begin{proof}
  The proof is similar to the one of Proposition~\ref{delta-s-z-k-conv}, but there are some
  differences we need to deal with. First, as in the proof of Proposition~\ref{delta-s-z-k-conv}, we choose $\e' \in (0,\frac{\eps}{2})$ small enough, so that (\ref{eq:distance-sigma-ineq}) holds. Similarly as in (\ref{eq:delta-s-splitting}) we can then write
\begin{equation}
  \label{eq:I-s-splitting}
G_k^2(\sigma,r)= c_{N,s} \Bigl( \widetilde A_k(\sigma,r)+ \widetilde B_k(\sigma,r) \Bigr)
\end{equation}
with
$$
\widetilde A_{k}(\sigma,r) := \int_{ \Psi_\s((-\e,\e) \times B)} \frac{(u({\Psi}(\sigma,\frac{r}{k}))-u(y))(\z(r)-\z_k(y))}{|\Psi(\s,\frac{r}{k}) - y|^{N+2s}}\, dy
$$
and 
$$
\widetilde B_k(\sigma,r) = \int_{\R^N \setminus \Psi_\s((-\e,\e) \times B)} \frac{(u({\Psi}(\sigma,\frac{r}{k}))-u(y))(\z(r)-\z_k(y))}{|\Psi(\s,\frac{r}{k}) - y|^{N+2s}}\, dy.
$$
As noted in the proof of Lemma~\ref{B-k-new-lemma}, we have 
$$
|\Psi(\s,\frac{r}{k})- y| \ge \frac{|\sigma - y|}{3} +\eps' \quad \text{for $y \in \R^N \setminus \Psi_\s((-\e,\e) \times B)$, $0<r< k \e'$.}
$$
Therefore, since $u \in L^\infty(\R^N)$, we may estimate as in the proof of Lemma~\ref{B-k-new-lemma} to get  
$$
|\widetilde B_k(\sigma,r)| \le 2 \|u\|_{L^\infty} \int_{\R^N \setminus \Psi_\s((-\e,\e) \times B)}\frac{|\rho(r)-\rho_k(y)|}{|\Psi(\s,\frac{r}{k}) - y|^{N+2s}}dy
\le C \Bigl(|\rho(r)| + k^{-1}\Bigr).
$$
Here, as before, the letter $C$ stands for various positive constants. Consequently,
\begin{equation}
  \label{eq:tilde-B-estimate-limit}
\lim_{k \to \infty}k^{-s}|\widetilde B_k(\sigma,r)|=0 \qquad \text{for every $\sigma \in \Omega$, $r\ge 0$,}
\end{equation}
since $\rho$ has compact support in $\R$, and
\begin{equation}
  \label{eq:tilde-B-estimate-majorization}
k^{-s}|\widetilde B_k(\sigma,r)| \le C k^{-s}\Bigl(|\rho(r)| + k^{-1}\Bigr) \le \frac{C}{1+r^{1+s}}  
\qquad \text{for $k \in \N$, $0 \le r < k \eps'$, $\sigma \in \partial \Omega$.}
\end{equation}
Hence it remains to estimate $\widetilde A_{k}(\sigma,r)$. For this we note that, by the same change of variables as in (\ref{change-of-variables-A}), we have 
\begin{align}
&\widetilde A_{k}(\sigma,r)= \int_{-\e}^{\e} \int_{B}{\rm Jac}_{\Psi_\s}(z,\tilde r)\: \frac{(u(\Psi(\frac{r}{k},0))-u(\Psi_\s(\tilde r,z)))  
(\z(r)-\z(k \tilde r))}{|\Psi_\s(\frac{r}{k},0) - \Psi_\s(\tilde r,z)|^{N+2s}}\,d z d \tilde r \nonumber\\
&= k^s \int_{\R}\frac{\z(r)-\z(r+t)}{|t|^{1+s}} \widetilde \cK_k(r,t)dt  \label{tilde-A-change-of-variables}
\end{align}
with the kernel 
\begin{align*}
&\widetilde \cK_k(r,t)\\
&= \left\{
  \begin{aligned}
   & \Bigl(\frac{|t|}{k}\Bigr)^{N+s} \int_{\frac{k}{|t|}B} \frac{\bigl(u(\Psi_\s(\frac{r}{k},0))-u(\Psi_\s(\frac{r+t}{k},\frac{|t|}{k}z))\bigr) {\rm Jac}_{{\Psi_\s}}(\frac{r+t}{k},\frac{|t|z}{k})}{\bigl|{\Psi_\s}(\frac{r}{k},0)-{\Psi_\s}(\frac{r+t}{k},\frac{|t|z}{k}) \bigr|^{N+2s}}dz,&&\quad t\in (-k\eps-r,k\eps-r), \\
   &0,&&\quad t \not \in (-k\eps-r,k\eps-r).
    \end{aligned}
\right.
\end{align*}
Since $u\in   C^s(\R^N)$ and $\Psi_\s$ is Lipschitz, we have  
$$
\bigl|u(\Psi_\s(\frac{r}{k},0))-u(\Psi_\s(\frac{r+t}{k},\frac{|t|}{k}z)\bigr| \leq C  \Bigl(\bigl(\frac{|t|}{k}\bigr)^2 +\bigl(\frac{|t z|}{k}\bigr)^{2}\Bigr)^{\frac{s}{2}} \leq C\Bigl(\frac{|t|}{k}\Bigr)^s(1+|z|^s),
$$
for $\sigma \in \partial \Omega$, $r \in (-k\eps,k\eps)$, $t \in (-k\eps-r,k\eps-r) \setminus \{0\}$ and $z \in \frac{k}{|t|}B$. Therefore, by using Lemma~\ref{new-lemma-Psi-sigma}(i),(iv) as in \eqref{eq:est-cK1}, 
\begin{equation}
  \label{eq:tilde-K-bound}
|\widetilde \cK_k(r,t)|   \le  C \int_{\R^{N-1}}(1+|z|^s) (1+|z|^2)^{-\frac{N+2s}{2}}dz \le C \int_{\R^{N-1}}(1+|z|)^{-N-s}dz < \infty.
\end{equation}
Inserting this estimate in (\ref{tilde-A-change-of-variables}), we conclude that 
$$
k^{-s} |\widetilde A_{k}(\sigma,r)|\le C \int_{\R}\frac{|\z(r)-\z(r+t)|}{|t|^{1+s}}\,dt =  C \int_{\R}\frac{|\rho(r)-\rho(r+t)|}{|t|^{1+s}}\,dt \le  \frac{C}{1+r^{1+s}}.
$$
for $k \in \N$, $0 \le r < k \eps'$, $\sigma \in \partial \Omega$. Combining this inequality with (\ref{eq:tilde-B-estimate-majorization}), we obtain (\ref{eq:I-s-z-k-concl-bound}). 
Moreover, since $u \in C^s_0(\ov{\O})$ and $\psi= \frac{u}{\delta^s} \in C^0(\ov\O)$, we have 
\be\label{eq:equzLemf}
\lim_{k \to \infty} k^s \left[ u(\Psi_\s(\frac{r}{k},0))-u(\Psi_\s(\frac{r+t}{k},\frac{|t|}{k}z))\right]  =  \psi(\s)(r^s_+ - (r+t)^s_+ )
\ee
for $\sigma \in \partial \Omega$, $r> 0$ and $t \in \R$ and $z \in \R^{N-1}$. Consequently, arguing as for \eqref{eq:est-cK2} with Lemma~\ref{new-lemma-Psi-sigma}(i)(ii),(iv),(v) and the dominated convergence theorem, we find that  
\begin{equation}
  \label{eq:w-k(r-t-z)-limit}
\lim_{k \to \infty}  \widetilde \cK_k(r,t) = \psi(\s)\frac{(r^s_+ - (r+t)^s_+)}{|t|^s}\int_{\R^{N-1}} (1+|z|^2)^{-\frac{N+2s}{2}} dz = a_{N,s}\psi(\s)\frac{(r^s_+ - (r+t)^s_+)}{|t|^s}  
\end{equation}
for $\sigma \in \partial \Omega$, $r> 0$ and $t \in \R$ with $a_{N,s}$ given in \eqref{def-a-n-s}. Hence, by (\ref{tilde-A-change-of-variables}), (\ref{eq:tilde-K-bound}), (\ref{eq:w-k(r-t-z)-limit}) and the dominated convergence theorem,
$$
\lim_{k \to \infty} k^{-s} \widetilde A_{k}(\sigma,r)= a_{N,s} \psi(\s) \int_{\R}\frac{ (r^s_+ - (r+t)^s_+ )(\z(r)-\z(r+t))}{|t|^{1+2s}}dt = \frac{a_{N,s}}{c_{1,s}} \psi(\s) \tilde I(r) = \frac{\psi(\s) \tilde I(r)}{c_{N,s}},
$$
where we used again \eqref{eq:bNs-aNs} for the last equality. Combining this with (\ref{eq:I-s-splitting}) and (\ref{eq:tilde-B-estimate-limit}), we obtain~(\ref{eq:I-s-z-k-concl-limit}). 
\end{proof}
 
We are now ready to complete the 
 \begin{proof}[Proof of Proposition~\ref{lem:lim-ok-new-section}]
Combining (\ref{eq:h-k-0-bound}), (\ref{eq:h-2-bound}) and (\ref{eq:I-s-z-k-concl-bound}), we see that there exists $\eps'>0$ with the property that the functions $G_k$ defined in (\ref{eq:def-ge-k}) satisfy 
\begin{equation}
  \label{eq:complete-proof-bound}
\frac{G_k(\s,r)}{k} \le C \frac{r^{s-1} + r^{s-1+\alpha}}{1+r^{1+s}}\qquad \text{for $k \in \N$, $0 \le r < k \eps'$}
\end{equation}
with a constant $C>0$ independent of $k$ and $r$. Since $s, \alpha \in (0,1)$, the RHS of this inequality is integrable over $[0,\infty)$. Moreover, by (\ref{eq:h-k-0-limit}), (\ref{eq:h-2-limit}) and (\ref{eq:I-s-z-k-concl-limit}), 
\begin{equation}
  \label{eq:complete-proof-limit}
\frac{1}{k}G_k(\s,r) \to  [X(\s) \cdot \nu(\s)] \psi^2(\s)h'(r)\bigl(r^{s}(-\Delta)^s \z(r)-\tilde I(r)\bigr)
\end{equation}
for every $r>0$, $\s \in\de \Omega$ as $k \to \infty$. Next we note that, by a standard computation, 
\begin{equation}
  \label{eq:complete-proof-limit-1}
(-\Delta)^s  h(r)= (-\Delta)^s [r^{s}_+ \z(r)] = \z(r)(-\Delta)^s r^s_+ + r^s_+  (-\Delta)^s \z(r)-\tilde I(r) = r^s_+  (-\Delta)^s \z(r)-\tilde I(r)
\end{equation}
for $r>0$ since $r^s_+$ is an $s$-harmonic function on $(0,\infty)$ see e.g \cite{BC}.  Hence, by (\ref{claim-first-reduction}),~(\ref{claim-first-reduction}), (\ref{eq:complete-proof-bound}), (\ref{eq:complete-proof-limit}), (\ref{eq:complete-proof-limit-1}) and the dominated convergence theorem, we conclude that 
 \begin{align*}
\lim_{k \to \infty}\int_{\Omega}g_k dx &= \int_0^\infty h'(r)(-\Delta)^s  h(r)dr  \int_{\partial \Omega}[X(\s) \cdot \nu(\s)] \psi^2(\s)d\s\\
 &= \int_{\R} h'(r)(-\Delta)^s  h(r)dr  \int_{\partial \Omega}[X(\s) \cdot \nu(\s)] \psi^2(\s)d\s,
 \end{align*}
as claimed in (\ref{eq:lem:new-section-claim}). 
 \end{proof}
 
\appendix
\section{}

Here we give a short proof of the uniqueness of positive minimizers of the problem \eqref{eq:def-lambda-sp} for $1 \le p \le 2$.

\begin{lemma}
\label{uniqueness-extended}  
Let $\Omega \subset \R^N$ be a bounded open set of class $C^{1,1}$, let $p\in[1,2]$, and let $u_1$ and $u_2$ be two positive minimizers of \eqref{eq:def-lambda-sp}. Then $u_1=u_2$.
\end{lemma}

\begin{proof}
 Suppose by contradiction that there are two different positive minimizers $u_1,u_2$ for 
the minimization problem. Then, since $\|u_1\|_{L^p(\O)} = \|u_2\|_{L^p(\O)} = 1$, the difference $u_1-u_2$ changes sign. Since moreover $\frac{u_1}{\delta^s}$ and $\frac{u_2}{\delta^s}$ are continuous positive functions on $\overline \Omega$ by Lemma~\ref{reg-prop-minimizers}, there exists a maximal $\tau \in (0,1)$ with
$$
\tau u_1 \le u_2 \qquad \text{on $\overline \Omega$.}
$$
Moreover, $\tau u_1 \not \equiv u_2$ since $u_1-u_2$ changes sign. Consequently, $v:= u_2 -\tau u_1$ satisfies $v \ge 0$ on $\overline \Omega$ and $v \not \equiv 0$. Moreover, using that $p-1 \in [0,1]$ and $\tau \in (0,1)$, we find that 
$$
(-\Delta)^s v = \lambda \bigl(u_2^{p-1}-\tau u_1^{p-1}\bigr) \ge \lambda \bigl(u_2^{p-1}-(\tau u_1)^{p-1}\bigr) \ge 0 \quad \text{in $\Omega$,}\qquad v= 0 \qquad \text{in $\R^N \setminus \Omega$}
$$
with $\lambda := \lambda_{s,p}(\Omega)>0$. Now the strong maximum principle for the fractional Laplacian and the fractional Hopf lemma implies that $v= u_2- \tau u_1$ is strictly positive in $\Omega$ and $\frac{v}{\delta^s}>0$ on $\partial \Omega$. This contradicts the maximality of $\tau$. Hence uniqueness holds.
\end{proof}

\end{document}